\documentclass{article}

\input{epsf.tex}
\usepackage{amsfonts,amssymb,verbatim,amsmath,amsthm,latexsym,textcomp,amscd}
\usepackage{latexsym,amsfonts,amssymb,epsfig,verbatim}
\usepackage{amsmath,amsthm,amssymb,latexsym,graphics,textcomp}
\usepackage{graphicx}
\usepackage{color}
\usepackage{url}

\input xy
\xyoption{all}

\theoremstyle{plain}
\newtheorem{theorem}{Theorem}[section]

\newtheorem{lemma}[theorem]{Lemma}
\newtheorem{corollary}[theorem]{Corollary}
\newtheorem{remark}[theorem]{Remark}

\newtheorem{proposition}[theorem]{Proposition}

\newtheorem*{claim}{Claim}

\def\Mod{\mbox{\rm{Mod}}}
\def\ML{{\mathcal {ML}}}
\def\PML{{\mathbb P}{\mathcal {ML}}}
\def\FL{{\mathcal {FL}}}
\def\EL{{\mathcal {EL}}}
\def\PFL{{\mathbb P}{\mathcal {FL}}}
\def\T{{\mathcal T}}

\def\S{{\mathcal S}}

\def\C{{\mathcal C}}
\def\G{{\mathcal G}}

\def\Diff{\mbox{\rm{Diff}}}
\def\id{\mbox{\rm{Id}}}

\def\H{{\mathcal H}}

\def\X{{\mathfrak X}}

\DeclareMathOperator{\ev}{ev}

\def\Ann{{\text{Ann}}}
\def\Min{{\text{Min}}}
\def\Cur{{\text{Cur}}}
\def\FL{{\mathcal {FL}}}
\def\PFL{{\mathbb P}{\mathcal {FL}}}
\def\sys{{\text{sys}}}
\def\INT{{\text{int}}}

\newcommand{\ATF}{{\mathbb A}}

\newcommand{\I}{i}

\newcommand{\co}{{\colon}}

\title{The universal Cannon--Thurston map and the boundary of the curve complex}
\author{Christopher J. Leininger\thanks{partially supported by NSF grant DMS-0603881}, Mahan Mj\thanks{partially supported by a UGC Major Research Project grant} and Saul Schleimer\thanks{partially
supported by NSF grant DMS-0508971}}

\begin{document}

\maketitle

\begin{abstract}
In genus two and higher, the fundamental group of a closed surface acts naturally on the curve complex of the surface
with one puncture.  Combining ideas from previous work of Kent--Leininger--Schleimer and Mitra, we construct a
universal Cannon--Thurston map from a subset of the circle at infinity for the closed surface group onto the boundary
of the curve complex of the once-punctured surface.  Using the techniques we have developed, we also show that the
boundary of this curve complex is locally path-connected.

\smallskip

\begin{center}

{\em AMS subject classification =   20F67(Primary), 22E40   57M50}

\end{center}

\end{abstract}

\tableofcontents

%%%%%%%%%%%%%%%%%%%%%%
\section{Introduction.}
%%%%%%%%%%%%%%%%%%%%%%

\subsection{Statement of results.}
%In this paper we prove, combining ideas from previous work of Kent-Leininger-Schleimer \cite{kls} and Mitra \cite{mitra-ct},
%the existence of  a Universal Cannon--Thurston map from a subset of the circle onto
%the boundary of the curve complex of a surface with one puncture. Further, we
%show that the boundary of the curve complex of a surface with one puncture is locally
%path-connected.

Fix a hyperbolic metric on a closed surface $S$ of genus at least two.  This identifies the universal cover with the
hyperbolic plane $p:\mathbb H \to S$. Fix a basepoint $z \in S$ and a point $\widetilde z \in p^{-1}(z)$.  This defines
an isomorphism between the group $\pi_1(S,z)$ of homotopy classes of loops based at $z$ and the group $\pi_1(S)$ of
covering transformations of $p:\mathbb H \to S$.

We will also regard the basepoint $z \in S$ as a marked point on $S$.   As such, we write $(S,z)$ for the surface $S$
with the marked point $z$.
We could also work with the punctured surface $S - \{z\}$; however a marked point is more convenient for us.

Let $\C(S)$ and $ \C(S,z)$ denote the curve complexes of $S$ and $(S,z)$ respectively and let $\Pi:\C(S,z) \to \C(S)$
denote the forgetful projection.  From \cite{kls}, the fiber over $v \in \C^0(S)$ is $\pi_1(S)$--equivariantly
isomorphic to the Bass-Serre tree $T_v$ determined by $v$.  The action of $\pi_1(S)$ on $\C(S,z)$ comes from the
inclusion into the mapping class group $\Mod(S,z)$ via the Birman exact sequence (see Section \ref{S:mod(S)}). We
define a map
$$\Phi:\C(S) \times \mathbb H \to \C(S,z)$$ by sending $\{ v\} \times \mathbb H$  to $T_v \cong \Pi^{-1}(v) \subset \C(S,z)$ in a $\pi_1(S)$--equivariant way and then extending over simplices using barycentric coordinates (see Section \ref{S:construct Phi}).

Given $v \in \C^0(S)$, let $\Phi_v$ denote the restriction to $\mathbb H \cong \{v\} \times \mathbb H$
$$\Phi_v:\mathbb H \to \C(S,z).$$

Suppose that $r \subset \mathbb H$ is a geodesic ray that eventually lies in the preimage of some proper essential
subsurface of S.  We will prove in Section 3 that $\Phi_v(r) \subset \C(S, z)$ has finite diameter. The remaining rays
define a subset $\ATF \subset \partial \mathbb H$ (of full Lebesgue measure). Our first result is the following.

\begin{theorem} [Universal Cannon--Thurston map] \label{cannonthurston}
For any $v \in \C^0(S)$, the map $\Phi_v: \mathbb H \to \C(S,z)$ has a unique continuous $\pi_1(S)$--equivariant extension
$$\overline\Phi_v: \mathbb H \cup \ATF \to \overline \C(S,z).$$
The map $\partial \Phi = \overline \Phi_v|_{\ATF}$ does not depend on $v$ and is a quotient map onto $\partial
\C(S,z)$.  Given distinct points $x,y \in \ATF$, $\partial \Phi(x) = \partial \Phi(y)$ if and only if $x$ and $y$ are
ideal endpoints of a leaf (or ideal vertices of a complementary polygon) of the lift of an ending lamination on $S$.
\end{theorem}

We recall that a Cannon--Thurston map was constructed by Cannon and Thurston \cite{CT} for the fiber subgroup of the
fundamental group of a closed hyperbolic $3$--manifold fibering over the circle.  The construction was then extended to
simply degenerate, bounded geometry Kleinian closed surface groups by Minsky \cite{minsky-top}, and in the general
simply degenerate case by the second author \cite{mahan-amalgeo},\cite{mahan-split}.  In all these cases, one produces
a quotient map from the circle $\partial \mathbb H$ onto the limit set of the Kleinian group $\Gamma$. In the quotient,
distinct points are identified if and only if they are ideal endpoints of a leaf (or ideal vertices of a complementary
polygon) of the lift of an ending lamination for $\Gamma$.  This is either one or two ending laminations depending on
whether the group is singly or doubly degenerate; see \cite{mahan-elct}.

In a similar fashion, the second author \cite{mitra-endlam} has constructed a Cannon--Thurston map for any
$\delta$--hyperbolic extension $\Gamma$ of a group $G$ by $\pi_1(S)$,
\[ 1 \to \pi_1(S) \to \Gamma \to G \to 1\]
(for a discussion of such groups see \cite{mosher-hbh,FMcc}).  This is a $\pi_1(S)$--equivariant quotient map from
$\partial \mathbb H$ onto the Gromov boundary of $\Gamma$.  As above, the quotient identifies distinct points if and
only if they are ideal endpoints of a leaf (or ideal vertices of a complementary polygon) of the lift of an ending
lamination for $G$.

The map $\partial \Phi$ is universal in that distinct points are identified if and only if they are the ideal endpoints
of a leaf (or ideal vertices of a complementary polygon) of the lift of \textit{any} ending lamination on $S$.  We
remark that the restriction to $\ATF$ is necessary to get a reasonable quotient: the same quotient applied to the
entire circle $\partial \mathbb H$ is a non-Hausdorff space.

It follows from the above description of the various Cannon--Thurston maps that the universal property of $\partial
\Phi$ can also be rephrased as follows.  If $F:\partial \mathbb H \to \Omega$ is any Cannon--Thurston map as
above---so, $\Omega$ is either the limit set of a Kleinian group, or the Gromov boundary of a hyperbolic extension
$\Gamma$---then there exists a map
\[ \phi_F: F(\ATF) \to \partial \C(S,z) \]
so that $ \phi_F \circ F|_{\ATF} = \partial \Phi$.  Moreover, because $\partial \Phi$ identifies \textit{precisely} the
required points to make this valid, one sees that any $\pi_1(S)$--equivariant quotient of $\ATF$ with this property is
actually a $\pi_1(S)$--equivariant quotient of $\partial \C(S,z)$.\\

It is a classical fact, due to Nielsen, that the action of $\pi_1(S)$ on $\partial \mathbb H$ extends to the entire
mapping class group $\Mod(S,z)$. It will become apparent from the description of $\ATF$ given below that this
$\Mod(S,z)$ action restricts to an action on $\ATF$.  In fact, we have

\begin{theorem} \label{full equivariance}
The quotient map
\[ \partial \Phi: \ATF \to \partial \C(S,z) \]
is equivariant with respect to the action of $\Mod(S,z)$.
\end{theorem}

As an application of the techniques we have developed, we also prove the following.

\begin{theorem} \label{localpathconn}
The Gromov boundary $\partial \C(S,z)$ is path-connected and locally path-connected.
\end{theorem}

We remark that $\ATF$ is noncompact and totally disconnected, so unlike the proof of local connectivity in the Kleinian
group setting, Theorem \ref{localpathconn} does not follow immediately from Theorem \ref{cannonthurston}.

This strengthens the work of the first and third authors in \cite{leinsch} in a special case: in \cite{leinsch} it was
shown that the boundary of the curve complex is connected for surfaces of genus at least $2$ with any positive number
of punctures and closed surfaces of genus at least $4$.  The boundary of the complex of curves describes the space of
simply degenerate Kleinian groups as explained in \cite{leinsch}.  These results seem to be the first ones providing
some information about the \textit{topology} of the boundary of the curve complex.  The question of connectivity of the
boundary was posed by Storm, and the general problem of understanding its topology by Minsky in his 2006 ICM address.
Gabai \cite{gabai} has now given a proof of Theorem \ref{localpathconn} for all surfaces $\Sigma$ for which
$\C(\Sigma)$ is nontrivial, except the torus, $1$--punctured torus and $4$--punctured sphere, where it is known to be
false.

\bigskip

\noindent
\textbf{Acknowledgements.}  The authors wish to thank the Mathematical Sciences Research Institute for its hospitality
during the Fall of 2007 where this work was begun.  We would also like to thank the other participants of the two
programs, \textit{Kleinian Groups and Teichm\"uller Theory} and \textit{Geometric Group Theory}, for providing a
mathematically stimulating and lively atmosphere.

%%%%%%%%%%%%%%%%%%%%%%%%%%%%%%%%%%%%%%%%%%%%%%%%%%%%%%%%%%%
\subsection{Notation and conventions.} \label{S:notation}
%%%%%%%%%%%%%%%%%%%%%%%%%%%%%%%%%%%%%%%%%%%%%%%%%%%%%%%%%%

\subsubsection{Laminations.} \label{S:laminations}
For a discussion of laminations, we refer the reader to \cite{harer-penner}, \cite{CEG}, \cite{bonahon-curr-teich},
\cite{thurstonnotes}, \cite{CB}.

A \textit{measured lamination} on $S$ is a lamination with a transverse measure of full support. A measured lamination
on $S$ will be denoted $\lambda$ with the support---the underlying lamination---written $|\lambda|$. We require that
all of our laminations be essential, meaning that the leaves lift to quasigeodesics in the universal cover.

If $a$ is an arc or curve in $S$ and $\lambda$ a measured lamination, we write $\lambda(a) = \int_a d \lambda$ for the
total variation of $\lambda$ along $a$. We say that $a$ is transverse to $\lambda$ if $a$ is transverse to every leaf
of $|\lambda|$. If $v$ is the isotopy class of a simple closed curve, then we write
\[\I(v,\lambda) = \inf_{\alpha \in v} \lambda(\alpha) \]
for the intersection number of $v$ with $\lambda$, where $\alpha$ varies over all representatives of the isotopy class
$v$.

Two measured laminations $\lambda_0$ and $\lambda_1$ are \textit{measure equivalent} if for every isotopy class of
simple closed curve $v$, $\I(v,\lambda_0) = \I(v,\lambda_1)$. Every measured lamination is equivalent to a unique
\textit{measured geodesic lamination} (with respect to the fixed hyperbolic structure on $S$).  This is a measured
lamination $\lambda$ for which $|\lambda|$ is a geodesic lamination.  Given a measured lamination $\lambda$, we let
$\hat \lambda$ denote the measure equivalent measured geodesic lamination.  We will describe a preferred choice of
representative of the measure class of a measured lamination in Section \ref{S:pointpos} below.

We similarly define measured laminations on $(S,z)$ as compactly supported measured laminations on $S - \{z\}$.  In the
situations that we will be considering, these will generally not arise as geodesic laminations for a hyperbolic metric
on $S - \{z\}$, though any one is measure equivalent to a measured geodesic lamination for a complete hyperbolic metric
on $S - \{z\}$.
%When it is convenient, we will try to aid in the distinction by referring to measured laminations on $(S,z)$ by $\mu$, again with the support denoted $|\mu|$.

The spaces of (measure classes of) measured laminations will be denoted by $\ML(S)$ and $\ML(S,z)$. The topology on
$\ML$ is the weakest topology for which $\lambda \mapsto \I(v,\lambda)$ is continuous for every simple closed curve
$v$.  Scaling the measures we obtain an action of $\mathbb R^+$ on $\ML(S) - \{ 0 \}$ and $\ML(S,z) -
\{ 0 \}$, and we denote the quotient spaces $\PML(S)$ and $\PML(S,z)$, respectively.

A particularly important subspace is the space of \textit{filling laminations} which we denote $\FL$.  These are the
measure classes of measured laminations $\lambda$ for which all complementary regions of its support $|\lambda|$ are
disks (for $S-\{z\}$, there is also a single punctured disk).  The quotient of $\FL$ by forgetting the measures will be
denoted $\EL$ and is the space of \textit{ending laminations}.  For notational simplicity, we will denote the element
of $\EL$ associated to the measure class of $\lambda$ in $\FL$ by its support $|\lambda|$.

Train tracks provide another useful tool for describing measured laminations.  See \cite{thurstonnotes} and
\cite{harer-penner} for a detailed discussion of train tracks and their relation to laminations. We recall some of the
most relevant information.

A lamination $\mathcal L$ is carried by a train track $\tau$ if there is a map $f:S \to S$ homotopic to the identity
with $f(\mathcal L) \subset \tau$ so that for every leaf $\ell$ of $\mathcal L$ the restriction of $f$ to $\ell$ is an
immersion.  If $\lambda$ is a measured lamination carried by a train track $\tau$, then the transverse measure defines
weights on the branches of $\tau$ satisfying the switch condition---the sum of the weights on the incoming branches
equals the sum on the outgoing branches. Conversely, any assignment of nonnegative weights to the branches of a train
track satisfying the switch condition uniquely determines an element of $\ML$.  Given a train track $\tau$ carrying
$\lambda$, we write $\tau(\lambda)$ to denote the train track $\tau$ together with the weights defined by $\lambda$.

\begin{proposition}
Suppose that $\{ \lambda_n\}_{n=1}^\infty \cup \{\lambda\} \subset \ML$ are all carried by the train track $\tau$. Then
$\lambda_n \to \lambda$ if and only if the weights on each branch of $\tau$ defined by $\lambda_n$ converge to those
defined by $\lambda$.
\end{proposition}
\begin{proof} This is an immediate consequence of \cite[Theorem~2.7.4]{harer-penner}.
\end{proof}

There is a well-known construction of train tracks carrying a given lamination which will be useful for us.  For a
careful discussion, see \cite[Theorem 1.6.5]{harer-penner}, or \cite[Section 4]{brock-length-cont}. Starting with a
geodesic lamination $\mathcal L$ one chooses $\epsilon > 0$ very small and constructs a foliation, transverse to
$\mathcal L$, of the $\epsilon$--neighborhood $N_\epsilon(\mathcal L)$.  The leaves of this foliation are arcs called
\textit{ties}. Taking the quotient by collapsing each tie to a point produces a train track $\tau$ on $S$; see Figure
\ref{nbhd}.

\begin{figure}[htb]
\centerline{}\centerline{}
\begin{center}
\ \psfig{file=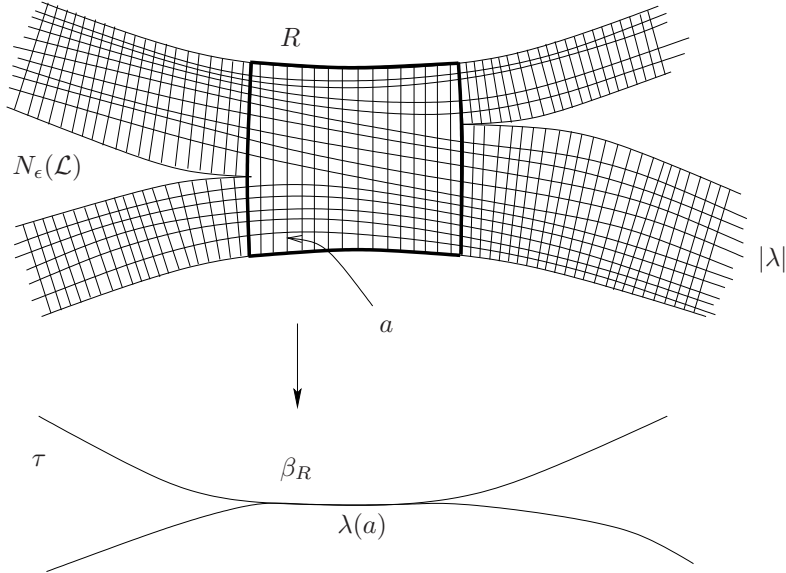,height=3truein} \caption{The train track $\tau$ is obtained by collapsing the ties of
$N_\epsilon(\mathcal L)$.  The lamination $\lambda$ defines weights on $\tau$: the weight on $\beta_R$ is $\lambda(a)$,
where $a$ is a tie in the rectangle $R$.} \label{nbhd}
\end{center}
  \setlength{\unitlength}{1in}
  \begin{picture}(0,0)(0,0)
    \put(2.42,2.2){$a$}
    \put(4.4,2.55){$|\lambda|$}
    \put(.5,3.03){$N_\epsilon(\mathcal L)$}
    \put(.6,1.5){$\tau$}
    \put(1.9,3.7){$R$}
    \put(1.9,1.45){$\beta_R$}
    \put(2.2,1.15){$\lambda(a)$}
  \end{picture}
\end{figure}

%\begin{figure}[htb]
%\centerline{} \centerline{}
%\begin{center}
%\ \psfig{file=tienbhd2.eps,height=1.5truein} \caption{A train track $\tau$ constructed from some $N_\epsilon(\mathcal
%L)$.} \label{nbhd}
%\end{center}
%  \setlength{\unitlength}{1in}
%  \begin{picture}(0,0)(0,0)
%    \put(1.2,1.3){$z_i$}
%    \put(3.37,1.28){$z_i$}
%  \end{picture}
%\end{figure}

We can view $N_\epsilon(\mathcal L)$ as being built from finitely many rectangles, each foliated by ties, glued
together along arcs of ties in the boundary of the rectangle.  In the collapse each rectangle $R$ projects to a branch
$\beta_R$ of $\tau$.  When $\tau$ is trivalent we may assume that $\tau \subset S$ is contained in $N_\epsilon(\mathcal
L)$, transverse to the foliation by ties, and the branch $\beta_R$ is contained in the rectangle $R$.

Suppose now that $\lambda$ is any measured lamination with $|\lambda| \subset N_\epsilon(\mathcal L)$, and $|\lambda|$
transverse to the ties.  If $R$ is a rectangle and $a$ a tie in $R$, then the weight on the branch $\beta_R$, defined
by $\lambda$, is given by $\lambda(a) = \int_a d\lambda$; see Figure \ref{nbhd}.

%\begin{figure}[htb]
%\centerline{}
%\centerline{}
%\begin{center}
%\ \psfig{file=weights1.eps,height=.75truein}
%\caption{$|\lambda|$ in $R$ and the weight on $\beta_R$ determined by $\lambda$.}
%\label{weights}
%\end{center}
%  \setlength{\unitlength}{1in}
%  \begin{picture}(0,0)(0,0)
%    \put(1.71,.58){$a$}
%    \put(.3,.85){$|\lambda|$}
%    \put(.5,1.3){$R$}
%    \put(3,1.2){$\beta_R$}
%    \put(3.5,.85){$\lambda(a)$}
%  \end{picture}
%\end{figure}

%%%%%%%%%%%%%%%%%%%%%%%%%%%%%%%%%%%%%%%%%%%%%%%%%%%%%%
\subsubsection{Mapping class groups.} \label{S:mod(S)}
%%%%%%%%%%%%%%%%%%%%%%%%%%%%%%%%%%%%%%%%%%%%%%%%%%%%%%

Recall that we have fixed a hyperbolic structure on $S$ as well as a locally isometric universal covering $p:\mathbb H
\to S$.  We also have a basepoint $\widetilde z \in p^{-1}(z)$ determining an isomorphism from $\pi_1(S)$, the covering group of
$p$, to $\pi_1(S,z)$, the group of homotopy classes of based loops.  All of this is considered fixed for
the remainder of the paper.

The \textit{mapping class group} of $S$ is the group $\Mod(S) = \pi_0(\Diff^+(S))$, where $\Diff^+(S)$ is the group of
orientation preserving diffeomorphisms of $S$. We define $\Mod(S,z)$ to be $\pi_0(\Diff^+(S,z))$, where $\Diff^+(S,z)$
is the group of orientation preserving diffeomorphisms of $S$ that fix
 $z$.

The evaluation map
\[\ev:\Diff^+(S) \to S \]
given by $\ev(f) = f(z)$ defines a locally trivial principal fiber bundle
\[ \Diff^+(S,z) \to \Diff^+(S) \to S.\]
A theorem of Earle and Eells \cite{earleeells} says that $\Diff_0(S)$, the component containing the identity, is contractible.  So
the long exact sequence of a fibration gives rise to Birman's exact sequence \cite{birmansequence,birmanbook}
\[
1 \rightarrow \pi_1(S) \rightarrow \Mod(S,z) \rightarrow \Mod(S) \to 1.
\]

%We use this sequence to view $\pi_1(S,z)$ as a subgroup of $\Mod(S,z)$.

%To describe the inclusion $\pi_1(S,z) \to \Mod(S,z)$ concretely, we first represent an element of $\pi_1(S,z)$ by a
%loop $\gamma$ based at $z$. Writing $\gamma \co [0,1] \to S$ with $\gamma(0) = \gamma(1) = z$, there is an isotopy
%$h_t\colon S \to S$ such that $h_1 = \id_S$ and $\gamma (t) = h_t(z)$ for all $t \in [0,1]$. Since $h_0(z) = z$, the
%map $h_1$ determines a mapping class in $\Mod(S,z)$, and this is the image of $\gamma$ in $\Mod(S,z)$ in the exact
%sequence. It is clear that the isotopy $h_t$ may be constructed so that the diffeomorphism $h_1$ is supported on any
%given neighborhood of the curve $\gamma \subset S$. For clarity, we write $h_\gamma$ for the diffeomorphism or mapping
%class associated to $\gamma \in \pi_1(S,z)$.

We elaborate on the injection $\pi_1(S) \to \Mod(S,z)$ in Birman's exact sequence.  Let
\[ \Diff_B(S,z) = \Diff_0(S) \cap \Diff^+(S,z).\]
The long exact sequence of homotopy groups identifies $\pi_1(S) \cong \pi_0(\Diff_B(S,z))$. This isomorphism is
induced by a homomorphism
\[\ev_*:\Diff_B(S,z) \to \pi_1(S)\]
given by $\ev_*(h) = [\ev(h_t)]$ where $h_t$, $t \in [0,1]$, is an isotopy from $h$ to $\id_S$, and
$[\ev(h_t)]$ is the based homotopy class of $\ev(h_t) = h_t(z)$, $t \in [0,1]$.  To see that this is a homomorphism,
suppose $h,h' \in \Diff_B(S,z)$ and $h_t$ and $h_t'$ are paths from $h$ and $h'$ respectively to $\id_S$.  Write
$\gamma(t) = h_t(z)$ and $\gamma'(t) = h_t'(z)$.  There is a path $H_t$ from $h \circ h'$ to $\id_S$ given as
\[ H_t = \left\{ \begin{array}{ll} h_{2t} \circ h' & \mbox{ for } t \in [0,1/2] \\ h_{2t-1}' & \mbox{ for } t \in
[1/2,1].\end{array} \right.\]
Then $H_t(z)$ is the path obtained by first traversing $\gamma$ then $\gamma'$, while $H_0
= h \circ h'$ and $H_1 = \id_S$.  So, $\ev_*(h \circ h') = \gamma \gamma'$, and $\ev_*$ is the required homomorphism.

Given $h \in \Diff_B(S,z)$, we will write $\sigma_h$ for a loop (or the homotopy class) representing $\ev_*(h)$.
Similarly, we will let $h_\sigma$ denote the mapping class (or a representative homeomorphism) determined by $\sigma
\in \pi_1(S)$.  When convenient, we will simply identify $\pi_1(S)$ with a subgroup of $\Mod(S,z)$.

%%%%%%%%%%%%%%%%%%%%%%%%%%%%%%%%%%%%%%%%%%%%%%%%%
\subsubsection{Curve complexes.} \label{S:curvecx}
%%%%%%%%%%%%%%%%%%%%%%%%%%%%%%%%%%%%%%%%%%%%%%%%%

A closed curve in $S$ is \textit{essential} if it is homotopically nontrivial in $S$. We will refer to a closed curve
in $S - \{z\}$ simply as a closed curve in $(S,z)$, and will say it is \textit{essential} if it is homotopically
nontrivial {\em and} nonperipheral in $S - \{z\}$.  Essential simple closed curves in $(S,z)$ are isotopic if and only
if they are isotopic in $S- \{z\}$.

Let $\C(S)$ and $\C(S,z)$ denote the \textit{curve complexes} of $S$ and $(S,z)$, respectively; see \cite{harvey} and
\cite{masur-minsky}.  These are geodesic metric spaces obtained by isometrically gluing regular Euclidean simplices
with all edge lengths equal to one.  The following is proven in \cite{masur-minsky}.
\begin{theorem} [Masur-Minsky]
The spaces $\C(S)$ and $\C(S,z)$ are $\delta$-hyperbolic for some $\delta > 0$.
\end{theorem}

We will refer to a simplex $v \subset \C(S)$ or $u \subset \C(S,z)$ and not distinguish between this simplex and the isotopy class of
multicurve it determines.  Any simple closed curve $u$ in $(S,z)$ can be viewed as a curve in $S$ which we denote
$\Pi(u)$. This gives a well-defined ``forgetful'' map
\[
\Pi \co \C(S,z) \to \C(S)
\]
which is simplicial.

Given a multicurve $v \subset \C(S)$, unless otherwise stated, we assume that $v$ is realized by its geodesic
representative in $S$. Associated to $v$ there is an action of $\pi_1(S)$ on a tree $T_v$, namely, the Bass--Serre tree
for the splitting of $\pi_1(S)$ determined by $v$.  We will make use of the following theorem of \cite{kls}.

\begin{theorem} [Kent-Leininger-Schleimer] \label{T:maintree}
The fiber of $\Pi$ over a point $x \in \C(S)$ is $\pi_1(S)$--equivariantly homeomorphic to the tree $T_v$, where $v$ is the unique simplex containing $x$ in its interior.\qed
\end{theorem}

%%%%%%%%%%%%%%%%%%%%%%%%%%%%%%%%%%%%%%%%%%%%%%%%%%%%%%%%%%%%
\subsubsection{Measured laminations and the curve complex.} \label{mlandcc}
%%%%%%%%%%%%%%%%%%%%%%%%%%%%%%%%%%%%%%%%%%%%%%%%%%%%%%%%%%%%

The curve complex $\C$ naturally injects into $\PML$ sending a simplex $v$ to the simplex of measures supported on $v$.
We denote the image subspace $\PML_\C$.  We note that this bijection $\PML_\C \to \C$ is not continuous in either
direction.  We will use the same notation for a point of $\PML_\C$ and its image in $\C$.

In \cite{klarreich-el} Klarreich proved that $\partial \C \cong \EL$.  Therefore, if we define
\[ \PML_{\overline \C} = \PML_\C \cup \PFL \]
then there is a natural surjective
map
$$\PML_{\overline \C} \to \overline{\C}$$
extending $\PML_\C \to \C$.  The following is a consequence of Klarreich's work \cite{klarreich-el}, stated
using our terminology.
\begin{proposition}[Klarreich] \label{klarreich}
The natural map $\PML_{\overline \C} \to \overline{\C}$ is continuous at every point of $\PFL$.  Moreover, a sequence
$\{v_n\} \subset \C$ converges to $|\lambda| \in \EL$ if and only if every accumulation point of $\{v_n\}$ in $\PML$
has $|\lambda|$ as its support.
\end{proposition}
\begin{proof}
Theorem 1.4 of \cite{klarreich-el} implies that if a sequence $\{v_n\}$ converges in $\overline \C$ to $|\lambda|$,
then every accumulation point of $\{v_n\}$ in $\PML$ has $|\lambda|$ as its support.  We need only verify that if
$\lambda \in \PFL$ and every accumulation point $\lambda'$ in $\PML$ of a sequence $\{v_n\}$ has $|\lambda| =
|\lambda'|$ then $\{v_n\}$ converges to $|\lambda|$ in $\overline \C$.

To see this, let $\{X_n\} \subset \T$ be any sequence in the Teichm\"uller space $\T$ for which $v_n$ is the shortest
curve in $X_n$.  In particular $\ell_{X_n}(v_n)$ is uniformly bounded.  Since every accumulation point of $\{v_n\}$ is
in $\PFL$, it follows that $X_n$ exits every compact set and so accumulates only on $\PML$ in the Thurston
compactification of $\T$. Moreover, if $\lambda'$ is any accumulation point of $X_n$ in $\PML$, then
$i(\lambda',\lambda) = 0$, and so $|\lambda'| = |\lambda|$ since $\lambda$ is filling.

Now according to Theorem 1.1 of \cite{klarreich-el}, the map
\[\sys: \T \to \C\]
sending $X \in \T$ to any shortest curve in $X$ extends to
\[\overline{\sys}: \T \cup \PFL \to \overline \C \]
continuously at every point of $\PFL$.  It follows that
\[\lim_{n \to \infty} v_n  = \lim_{n \to \infty} \sys(X_n) = |\lambda|\]
in $\overline \C$ and we are done.
\end{proof}

%%%%%%%%%%%%%%%%%%%%%%%%%%%%%%%%%%%%%%%%%%%%%%%%%%%%%%
\subsubsection{Cannon--Thurston maps.} \label{S:ctmaps}
%%%%%%%%%%%%%%%%%%%%%%%%%%%%%%%%%%%%%%%%%%%%%%%%%%%%%%

Fix $X$ and $Y$ hyperbolic metric spaces, $F : Y \rightarrow X$ a continuous map, and $Z \subset
\partial Y$ a subset of the Gromov boundary.
A \textit{$Z$--Cannon--Thurston map} is a continuous extension $\overline F:Y \cup Z \to X \cup \partial X$ of $F$.
That is, $\overline F|_Y = F$. We will simply call $\overline F$ a Cannon--Thurston map when the set $Z$ is clear from
the context.  We sometimes refer to the restriction $\partial F = \overline F|_Z$ as a Cannon--Thurston map.

This definition is more general than that in \cite{mitra-ct} in the sense that here we require $F$ only to be
continuous, whereas in \cite{mitra-ct} it was demanded that $F$ be an embedding.  Also, we do not require $\overline F$
to be defined on all of $\overline Y = Y \cup \partial Y$.

To prove the existence of such a Cannon--Thurston map, we shall use the following obvious criterion:
\begin{lemma} \label{ct-crit}
Fix $X$ and $Y$ hyperbolic metric spaces, $F : Y \to X$ a continuous map and $Z \subset \partial Y$ a subset.  Fix a
basepoint $x \in X$.  Then there is a $Z$--Cannon--Thurston map $\overline F : Y \cup Z \to X \cup \partial X$ if and
only if for every $z \in Z$ there is a neighborhood basis $B_i \subset Y \cup Z$ of $z$ and a collection of uniformly
quasiconvex sets $Q_i \subset X$ with $F(B_i \cap Y) \subset Q_i$ and $d_X(x, Q_i) \to \infty$ as $i \to \infty$.
Moreover,
$$
\bigcap_i \overline{Q_i} = \bigcap_i \partial Q_i = \left\{ \overline F(z) \right\}
$$
determines $\overline F(z)$ uniquely. \qed
\begin{comment}
Let $X$ and $Y$ be hyperbolic metric spaces and $F : Y \rightarrow X$ be a continuous map and $Z \subset
\partial Y$. Then a $Z$-Cannon--Thurston map $\overline{F}$ exists if for every $y \in Z$, there exists a neighborhood basis
$\{B_i(y)\}_{i=1}^\infty$ of $y \in Y \cup Z$ and uniformly quasiconvex sets $Q_i(y) \subset X$ with $F(B_i(y) \cap Y)
\subset Q_i(y)$ for all $i$ and $d(x,Q_i) \to \infty$ as $i \to \infty$ for some basepoint $x \in X$.  Moreover,
$\overline{F}(y)$ is the unique point of intersection of the sets
\[ \bigcap_i \overline{Q_i(y)} = \bigcap_i \partial Q_i(y) = \{\overline{F}(y) \}.\]
\end{comment}
\end{lemma}

%%%%%%%%%%%%%%%%%%%%%%%%%%%%%%%%%%%%%%%%%%%%
\section{Point position.} \label{S:pointpos}
%%%%%%%%%%%%%%%%%%%%%%%%%%%%%%%%%%%%%%%%%%%%

We now describe in more detail the map
\[ \Phi:\mathbb \C(S) \times \mathbb H \to \C(S,z)\]
as promised in the introduction,
and explain how this can be extended continuously to $\overline{\C}(S) \times \mathbb H$.

%%%%%%%%%%%%%%%%%%%%%%%%%%%%%%%%%%%%%%%%
\subsection{A bundle over $\mathbb H$.} \label{S:bundle over H}
%%%%%%%%%%%%%%%%%%%%%%%%%%%%%%%%%%%%%%%%

The bundle determining the Birman exact sequence has a subbundle obtained by restricting $\ev$ to $\Diff_0(S)$:
\[
\xymatrix{ \Diff_B(S,z) \ar[r] & \Diff_0(S) \ar[r]^{\quad \ev} & S.\\}
\]
As noted before, Earle and Eells proved that $\Diff_0(S)$ is
contractible, and hence there is a unique lift
\[\widetilde{\ev}:\Diff_0(S) \to \mathbb H\]
with the property that $\widetilde{\ev}(\id_S) = \widetilde z$.

The map $\widetilde{\ev}$ can also be described as follows.  Any diffeomorphism $S \to S$ has a lift $\mathbb H \to
\mathbb H$, and the contractibility of $\Diff_0(S)$ allows us to coherently lift diffeomorphisms to obtain an injective
homomorphism $\Diff_0(S) \to \Diff(\mathbb H)$.  Then $\widetilde{\ev}$ is the composition of this homomorphism with
the evaluation map $\Diff(\mathbb H) \to \mathbb H$ determined by $\widetilde z$.

Since $p$ is a covering map, $\widetilde{\ev}$ is also a fibration.  Appealing to the long exact sequence of homotopy
groups again, we see that the fiber over $\widetilde z$ is precisely $\Diff_0(S,z)$. We record this in the following
diagram
\begin{equation} \label{big diagram}
\xymatrix{ & & \mathbb H \ar[d]^p\\
\Diff_B(S,z) \ar[r] & \Diff_0(S) \ar[r]^{\quad \ev} \ar[ur]^{\widetilde{\ev}} & S\\
\Diff_0(S,z). \ar[u] \ar[ur]}
\end{equation}

The group $\Diff_B(S,z)$ acts on $\Diff_0(S)$ on the left by
\[h \cdot f = f \circ h^{-1}\]
for $h \in \Diff_B(S,z)$ and $f \in \Diff_0(S)$. Also recall from Section 1.2.2 that $\pi_1(S) \cong
\pi_0(\Diff_B(S,z))$ with this isomorphism induced by a homomorphism
\[\ev_*:\Diff_B(S,z) \to \pi_1(S).\]
\begin{lemma} \label{L:tildeevequivariant}
The lift
\[ \widetilde{\ev}:\Diff_0(S) \to \mathbb H\]
is equivariant with respect to $\ev_*$.
\end{lemma}
\begin{proof}
We need to prove
\[ \ev_*(h) (\widetilde \ev(f)) = \widetilde \ev(f \circ h^{-1})\]
for all $f \in \Diff_0(S)$ and $h \in \Diff_B(S,z)$.  Observe that since $h(z) = z$ for every $h \in \Diff_B(S,z)$,
$\ev(f) = \ev(f \circ h^{-1})$ for every $f \in \Diff_0(S)$.  Therefore, since $\widetilde \ev$ is a lift of $\ev$ we
have
\[
p (\widetilde{\ev}(f)) = \ev(f) = \ev(f \circ h^{-1}) = p(\widetilde{\ev}(f \circ h^{-1}))
\]
and hence $\widetilde \ev(f)$ differs from $\widetilde \ev(f \circ h^{-1})$ by a covering transformation $\sigma \in \pi_1(S)$:
\[
\widetilde \ev(f \circ h^{-1}) = \sigma (\widetilde \ev(f)).
\]

The covering transformation $\sigma$ appears to depend on both $f$ and $h$.  However if $H_t$, $t \in [0,1]$, is a path
in $\Diff_B(S,z)$ from $h = H_0$ to $h' = H_1$ then $\widetilde \ev(f \circ H_t^{-1})$ is constant in $t$:  this can be
seen from the above description of $\widetilde{\ev}$ as the evaluation map on the lifted diffeomorphism group.  It
follows that $\sigma$ depends only on $f$ and the component of $\Diff_B(S,z)$ containing $h$. In fact, continuity of
$\widetilde \ev$ and connectivity of $\Diff_0(S)$ implies that $\sigma$ actually only depends on the component of
$\Diff_B(S,z)$ containing $h$, and not on $f$ at all.

We have
\[
\sigma(\widetilde z) = \sigma (\widetilde \ev(\id_S)) = \widetilde \ev(\id_S \circ h^{-1}) = \widetilde \ev(h^{-1}).
\]
So if $h_t$, $t \in [0,1]$, is a path in $\Diff_0(S)$ from $h = h_0$ to $\id_S = h_1$, then since $\ev_*(h) = \sigma_h$
where $\sigma_h$ is represented by the loop $h_t(z)$, $t \in [0,1]$, it follows that $\sigma_h^{-1}$ is represented by
the loop $h_t^{-1}(z)$, $t \in [0,1]$; see Section \ref{S:mod(S)}.

Now observe that $\widetilde \ev(h_t^{-1})$, $t \in [0,1]$, is a lift of the loop $h_t^{-1}(z)$, $t \in [0,1]$, to a
path from $\sigma(\widetilde z)$ to $\widetilde{z}$. Therefore, $\sigma_h^{-1}$ is $\sigma^{-1}$, and hence $\sigma =
\sigma_h = \ev_*(h)$.
\end{proof}

%%%%%%%%%%%%%%%%%%%%%%%%%%%%%%%%%%%%%%%%%%%%%%
\subsection{An explicit construction of $\Phi$.} \label{S:construct Phi}
%%%%%%%%%%%%%%%%%%%%%%%%%%%%%%%%%%%%%%%%%%%%%%

We are now ready to give an explicit description of the map $\Phi$.  We will define first a map
\[ \widetilde \Phi: \C(S) \times \Diff_0(S) \to \C(S,z) \]
and show that this descends to a map $\Phi : \C(S) \times \mathbb{H} \to \C(S,z)$ by composing with $\widetilde{\ev}$
in the second factor.

Recall that for every $v \in \C^0(S)$, we have realized $v$ by its geodesic representative.  We would like to simply
define
\[ \widetilde \Phi(v,f) = f^{-1}(v). \]
However, this is not a curve in $(S,z)$ when $f(z)$ lies on the geodesic $v$.  The map we define in the end will agree
with this when $f(z)$ is not too close to $v$, and it is helpful to keep this in mind when trying to make sense of the
actual definition of $\widetilde \Phi$.

To carry out the construction of $\widetilde \Phi$, we now choose $\{ \epsilon(v) \}_{v \in \C^0(S)} \subset \mathbb
R_+$ so that $N(v) = N_{\epsilon(v)}(v)$, the $\epsilon(v)$--neighborhood of $v$, has the following properties:
\begin{itemize}
\item $N(v) \cong \S^1 \times [0,1]$ and
\item $N(v) \cap N(v') = \emptyset$ if $v \cap  v' = \emptyset$.
\end{itemize}
Let $N^\circ (v)$ denote the interior of $N(v)$ and $v^\pm$ denote the boundary components of $N(v)$.

Given a simplex $v \subset \C(S)$ with vertices $\{v_0,...,v_k\}$ we consider the barycentric coordinates for points in
$v$:
\[ \left\{ \sum_{j=0}^k s_j v_j \, \left| \, \sum_{j=0}^k s_j = 1 \mbox{ and } s_j \geq 0 , \, \, \mbox{ for all } j=0,...,k  \right. \right\}. \]

To define our map
\[ \widetilde \Phi: \C(S) \times \Diff_0(S) \to \C(S,z) \]
we first explain how to define it for $(v,f)$ with $v$ a vertex of $\C(S)$.  If $f(z) \not \in  N^\circ(v)$, then we set
\[\widetilde \Phi(v,f) = f^{-1}(v)\]
as suggested above.

If $f(z) \in N^\circ(v)$, then $f^{-1}(v^+)$ and $f^{-1}(v^-)$ are nonisotopic curves in $(S,z)$. We will define
$\widetilde \Phi(v,f)$ to be a point on the edge between these two vertices of $\C(S,z)$, depending on the distance
from $f(z)$ to the two boundary components $v^+$ and $v^-$.  Specifically, set
\[ t = \frac{d(v^+,f(z))}{2 \epsilon(v)},\]
where $d(v^+,f(z))$ is the distance inside $N(v)$ from $f(z)$ to $v^+$, and define
\[
\widetilde \Phi(v,f) = t f^{-1}(v^+) + (1-t) f^{-1}(v^-)
\]
in barycentric coordinates on the edge $\left[ f^{-1}(v^+),f^{-1}(v^-) \right]$.

In general, for a point $(x,f) \in \C(S) \times \Diff_0(S)$ with $x = \sum_j s_j v_j \in v = \{v_0,...,v_k\}$ we
define $\widetilde \Phi(x,f)$ as follows.   As before, if $f(z) \not \in \cup_j N^\circ(v_j)$, then define
\[
\widetilde \Phi(x,f) = \sum_j s_j f^{-1}(v_j).
\]
Otherwise, $f(z) \in N^\circ(v_i)$ for exactly one $i \in \{ 0,...,k\}$.  Set
\[t = \frac{d(v_i^+,f(z))}{2 \epsilon(v_i)}\]
as above, and define
\[
\widetilde \Phi(x,f) = s_i \left( t f^{-1}(v_i^+) + (1-t)f^{-1}(v_i^-) \right) + \sum_{j \neq i} s_j f^{-1}(v_j).
\]

The group $\Diff_B(S,z)$ acts on $\C(S) \times \Diff_0(S)$, trivially in the first factor and as described in Section
\ref{S:bundle over H} in the second factor.  Of course, since $\Diff_B(S,z) < \Diff^+(S,z)$ projects into $\Mod(S,z)$
it also acts on $\C(S,z)$.  The map $\widetilde \Phi$ is equivariant: given $h \in \Diff_B(S,z)$, $f \in \Diff_0(S)$
and $v$ a vertex in $\C(S)$, provided $f(z) \not \in N^\circ(v)$ we have
\[\begin{array}{rcl}
\widetilde \Phi(h \cdot (v,f)) & = & \widetilde \Phi(v,f \circ h^{-1})\\
 & = & (f \circ h^{-1})^{-1}(v)\\
 & = & h \circ f^{-1}(v)\\
 & = & h \cdot (f^{-1}(v))\\
 & = & h \cdot \widetilde\Phi(v,f). \end{array} \] The general situation is similar, but notationally more complicated.

\begin{proposition} \label{phidescends}
The map $\widetilde \Phi$ descends to a map  $\Phi$ making the following diagram commute
\[
\xymatrix{ \C(S) \times \Diff_0(S) \ar[drr]^{\widetilde \Phi} \ar[d]_{\id_{\C(S)} \times \widetilde \ev}\\
 \C(S) \times \mathbb H \ar[rr]_{\Phi} & & \C(S,z).}\]
Moreover, $\Phi$ is equivariant with respect to the action of $\pi_1(S)$.
\end{proposition}
Here the action of $\pi_1(S)$ on $\C(S) \times \mathbb H$ is trivial on the first factor and the covering group
action on the second.
\begin{proof}
We suppose that $\widetilde \ev(f_0) = \widetilde \ev(f_1)$ and must show $\widetilde \Phi(x,f_0) = \widetilde
\Phi(x,f_1)$.

Appealing to diagram \eqref{big diagram} in Section \ref{S:bundle over H}, it follows that $f_0 = f_1 \circ h$ for some
$h \in \Diff_0(S,z)$.   We suppose that $\alpha$ is a simple closed curve on $S$ and $f_0(z) \not \in \alpha$. Then
$f_1(z) = f_1(h(z)) = f_0(z) \not \in \alpha$ and
\[ d(f_0(z),\alpha) = d(f_1(h(z)),\alpha) = d(f_1(z),\alpha). \]
Moreover, $f_0^{-1}(\alpha) = h^{-1} (f_1^{-1}(\alpha))$ and since $h^{-1}$ is isotopic to the identity in $(S,z)$, it
follows that $f_0^{-1}(\alpha)$ and $f_1^{-1}(\alpha)$ are isotopic in $(S,z)$.

Recall that the dependence of $\widetilde{\Phi}(x,f)$ on $f$ was only via certain isotopy classes $f^{-1}(\alpha)$ and
a single distance $d(v^+,f(z))$.  Since these data are the same for $f_0$ and $f_1$, it follows that
\[ \widetilde \Phi(x,f_0) = \widetilde \Phi(x,f_1) \]
and so $\widetilde \Phi$ descends to $\C(S) \times \mathbb H$ as required.

Lemma \ref{L:tildeevequivariant} implies that $\id_{\C(S)} \times \widetilde \ev$ is equivariant with respect to
$\ev_*$.  Thus, since $\widetilde \Phi$ is equivariant, so is $\Phi$.
\end{proof}

\begin{proposition} \label{prop: factors through trees}
Given $x \in \C(S)$, let $v \subset \C(S)$ be the simplex containing $x$ in its interior.  Then the restriction
\[ \Phi_x = \Phi|_{\{x\} \times \mathbb H}:\mathbb H \to \Pi^{-1}(x)\]
is obtained by first projecting to $T_v$, then composing with the equivariant homeomorphism $T_v \cong \Pi^{-1}(x)$ from Theorem \ref{T:maintree}.
\end{proposition}
\begin{proof}
Fix $x \in \C(S)$, the simplex $v = \{v_0,...,v_k\} \subset \C(S)$ containing $x$ in its interior, and write
\[ x = \sum_{i=1}^k s_i v_i \]
in terms of barycentric coordinates.

We note that the neighborhoods $N(v_i)$ determine a map from $\mathbb H$ to the Bass--Serre tree $T_v$ associated to
$v$ as follows.  We collapse each component $U$ of the preimage $p^{-1}(N(v_i))$ onto an interval, say $[0,1]$, by the
projection defined as the distance to the component of $p^{-1}(v_i^+)$ meeting $U$, multiplied by $1/(2
\epsilon(v_i))$.  If we further collapse each component of the complement of
\[p^{-1}(N(v_0) \cup \cdots \cup N(v_k)) \]
to a point, the quotient space is precisely $T_v$.

The map $\Phi_x$ is constant on the fibers of the projection to $T_v$. That is, $\Phi_x:\{x \} \times \mathbb H \to
\Pi^{-1}(x) \subset \C(S,z)$ factors through the projection to $T_v$
\[
\xymatrix@R=1pt{ \{x\} \times \mathbb H \ar[rr]^{\Phi_x} \ar[dr] & & \Pi^{-1}(x)\\
 & T_v \ar[ru] & }.\]
Moreover, the equivariance of $\Phi$ implies that
\[ T_v \to \Pi^{-1}(x) \]
is equivariant.  According to \cite{kls}, the edge and vertex stabilizers in the domain and range agree, and in fact this map is
the homeomorphism given by Theorem \ref{T:maintree}, as required.
\end{proof}

%%%%%%%%%%%%%%%%%%%%%%%%%%%%%%%%%%%%%%%%%%%%%%%
\subsection{A further description of $\C(S,z)$.}
%%%%%%%%%%%%%%%%%%%%%%%%%%%%%%%%%%%%%%%%%%%%%%%%

We pause here to give a combinatorial description of $\C(S,z)$ which will be useful later, but is also of interest in its own right.
Given any simplex $v \subset \C(S)$, the preimage of the interior of $v$ admits a $\pi_1(S)$--equivariant homeomorphism
\[ \Pi^{-1}(\INT(v)) \cong \INT(v) \times T_v\]
as can be seen from Theorem \ref{T:maintree}.  As is well-known, the edges of $T_v$ can be labeled by the vertices of
$v$. Now, if $\phi:v' \to v$ is the inclusion of a face, then there is a $\pi_1(S)$--equivariant quotient map
$\phi^*:T_v \to T_{v'}$ obtained by collapsing all the edges of $T_v$ labeled by vertices \textit{not} in $\phi(v')$
(compare \cite{guirardellevitt}, for example).  This provides a description of $\Pi^{-1}(v)$, the preimage of the
closed simplex, as a quotient
\[ \left. \left( \bigsqcup_{\phi:v' \to v} v' \times T_{v'}\right) \right/ \huge{\sim}.\]
Here the disjoint union is taken over all faces $\phi:v' \to v$ and the equivalence relation $\sim$ is defined by
\[ (\varphi(x),t) \sim (x,\varphi^*(t))\]
for every inclusion of faces $\varphi:v'' \to v'$ and every $x \in v''$, $t \in T_{v'}$.  Said differently, we take the
product $v \times T_v$ and for every face $\phi:v' \to v$, we glue $v \times T_v$ to $v' \times T_{v'}$ along $\phi(v')
\times T_v$ by $\phi^{-1} \times \phi^*$.

We can do this for all simplices, then glue them all together, providing the following useful description of $\C(S,z)$.
\begin{theorem} \label{glued cc}
The curve complex $\C(S,z)$ is $\pi_1(S)$--equivariantly homeomorphic to
\[ \left. \left( \bigsqcup_{v \subset \C(S)} v \times T_v \right) \right/ \sim.\]
Here the disjoint union is taken over all simplices $v \subset \C(S)$, and the equivalence relation is generated by
\[(\phi(x),t) \sim (x,\phi^*(t))\]
for all inclusions of faces $\phi:v' \to v$ all $x \in v'$ and all $t \in T_v$.\qed
\end{theorem}

%%%%%%%%%%%%%%%%%%%%%%%%%%%%%%%%%%%%%%%%%%%%%%%%
\subsection{Extending to measured laminations.}
%%%%%%%%%%%%%%%%%%%%%%%%%%%%%%%%%%%%%%%%%%%%%%%%

The purpose of this section is to modify the above construction of $\Phi$ to build a map
\[
\Psi \co \ML(S) \times \mathbb H \to \ML(S,z)
\]
and to prove that this is continuous at every point of $\FL(S) \times \mathbb H$; see Corollary \ref{psicont}.  We do
this by defining a map on $\ML(S) \times \Diff_0(S)$, and checking that it descends to $\ML(S) \times \mathbb H$.

Before we can begin, we must specify a particular realization for each element of $\ML(S)$ as a measured lamination. We
begin by realizing all elements as measured geodesic laminations (recall we denote these with a hat, $\hat \lambda$),
then replace all simple closed geodesic components of the support with appropriately chosen annuli.  We now explain
this more precisely and set some notation.

Given a measured geodesic lamination $\hat \lambda$, the support $|\hat \lambda|$ can be decomposed into a finite union
of pairwise disjoint minimal sublaminations; see \cite{CB}.  Write
\[ \hat \lambda = \Cur(\hat \lambda) + \Min(\hat \lambda),\]
where $|\Cur(\hat \lambda)|$ is the union of all simple closed geodesics in $|\hat \lambda|$ and $|\Min(\hat \lambda)|
= |\hat \lambda| - |\Cur(\hat \lambda)|$.  We construct a measured lamination $\lambda$ measure equivalent to $\hat
\lambda$ by taking
\[ \lambda = \Ann(\lambda) + \Min(\lambda),\]
where $\Min(\lambda) = \Min(\hat \lambda)$ and $\Ann(\lambda)$ is a measured lamination whose support is a foliation on
annular neighborhoods of $|\Cur(\hat \lambda)|$ defined as follows.

The sublamination $\Cur(\hat \lambda)$ can be further decomposed as $\Cur(\hat \lambda) = \sum_j t_j v_j$, where $t_j
v_j$ means $t_j$ times the transverse counting measure on the simple closed geodesic component $v_j$ of $|\Cur(\hat
\lambda)|$.  Then $|\Ann(\lambda)|$ is the disjoint union $\cup_j N(v_j)$, with each $N(v_j)$ given the foliation by
curves equidistant to $v_j$.  This foliation of $N(v_j)$ is assigned the transverse measure which is
$t_j/(2\epsilon(v_j))$ times the distance between leaves, and $\Ann(\lambda)$ is the sum of these measured laminations;
see Figure \ref{cartoon} for a cartoon depiction of $\hat \lambda$ and $\lambda$.  Choosing $\{\epsilon(v)\}$
sufficiently small it follows that $|\Ann(\lambda)| \cap |\Min(\lambda)| = \emptyset$ for all $\lambda$.

For future use, if $\Cur(\hat \lambda) = \sum_j t_j v_j$, then we define
\[ T(\hat \lambda) = T(\lambda) = \max_j t_j.\]
If $|\Cur(\hat \lambda)| = \emptyset$ we set $T(\hat \lambda) = T(\lambda) = 0$.

\begin{figure}[htb]
\centerline{}
\centerline{}
\begin{center}
\ \psfig{file=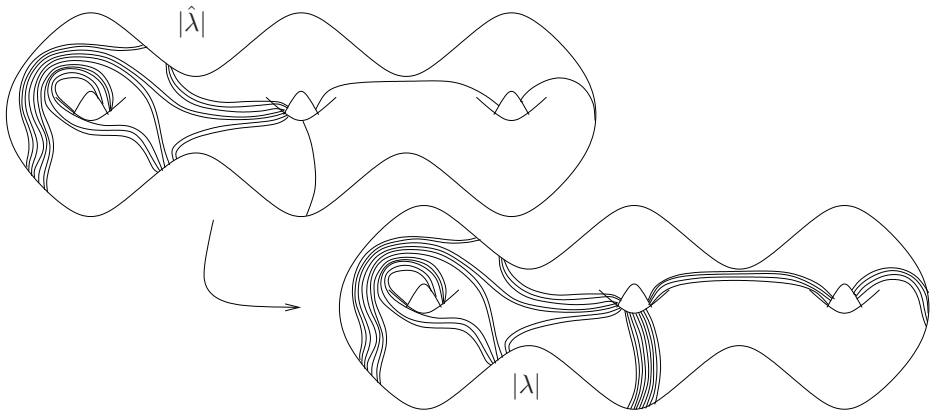,height=2.1truein}
\caption{Removing simple closed geodesics and inserting foliated annuli.}
\label{cartoon}
\end{center}
  \setlength{\unitlength}{1in}
  \begin{picture}(0,0)(0,0)
    \put(.9,2.6){$|\hat \lambda|$}
    \put(2.65,.7){$|\lambda|$}
  \end{picture}
\end{figure}

Whenever we refer to an element $\lambda$ of $\ML(S)$ in what follows, we will assume it is realized by such a measured
lamination. Of course $|\hat \lambda| \subset |\lambda|$, meaning that as subsets of $S$, $|\hat \lambda|$ is a subset of $|\lambda|$, \textit{and} that each leaf of $|\hat \lambda|$ is a leaf of $|\lambda|$.  The difference between the total variations assigned an arc
by $\lambda$ and $\hat \lambda$ is estimated by the following.
\begin{lemma} \label{notthatdifferent}
If $a$ is any arc transverse to $|\lambda|$, then it is also transverse to $|\hat \lambda|$ and we have
\[ |\lambda(a) - \hat \lambda(a)| \leq T(\hat \lambda).\]
\end{lemma}
\begin{proof} The transversality statement is an immediate consequence of $|\hat \lambda| \subset |\lambda|$.

Since $\Min(\lambda) = \Min(\hat \lambda)$, we see that
\[|\lambda(a) - \hat \lambda(a)| = |\Ann(\lambda)(a) - \Cur(\hat \lambda)(a)| \]
The intersection of $|\Ann(\lambda)| \cap a$ is a union of subarcs of $a$, each containing an intersection point of
$|\Cur(\hat \lambda)| \cap a$, with the possible exception of those arcs which meet the endpoints of $a$. If $a_0
\subset a$ is one of the subarcs which meets the boundary, then we have $|\Ann(\lambda)(a_0) - \Cur(\hat \lambda)(a_0)|
\leq T(\hat \lambda)/2$. Since there are at most $2$ such arcs, the desired inequality follows.
\end{proof}

The following is also useful.
\begin{lemma} \label{L:convergetolam}
Suppose $\lambda_n \to \lambda$ in $\ML(S)$ with $\lambda \in \FL(S)$.  Further suppose that $|\lambda_n|$ converges in
the Hausdorff topology on closed subsets of $S$ to a set $\mathcal L$.  Then $\mathcal L$ is a geodesic lamination
containing $|\lambda|$.
\end{lemma}
\begin{proof}
If $|\lambda_n| = |\hat \lambda_n|$ is a geodesic lamination for all $n$, then the fact that $\mathcal L$ is a geodesic
lamination is well-known; see \cite{CB}.

Since $\lambda_n \to \lambda$ and $\lambda \in \FL$, it follows that no simple closed geodesic occurs infinitely
often in $\{|\Cur(\hat \lambda_n)|\}$.  Further note that if $\{v_n\}$ is any sequence of distinct simple closed
geodesics in $S$, then their lengths tend to infinity and hence $\epsilon(v_n) \to 0$.  Therefore, the Hausdorff
distance between $|\hat \lambda_n|$ and $|\lambda_n|$ tends to zero, and so the Hausdorff limits of $|\hat \lambda_n|$ and
$|\lambda_n|$ are the same.  As above, we see that $\mathcal L$ is a geodesic lamination.
\end{proof}

Now, given any $(\lambda,f) \in \ML(S) \times \Diff_0(S)$, we would like to simply define
\[ \widetilde \Psi(\lambda,f) = f^{-1}(\lambda). \]
As before, this does not make sense when $f(z)$ lies on the supporting lamination $|\lambda|$.  This is remedied by
first splitting open the lamination along the leaf which $f(z)$ meets to produce a new measured lamination $\lambda'$
representing the measure class $\lambda$ (there is no ambiguity about how the measure is split since $\lambda$ has no
atoms).  The new lamination $|\lambda'|$ has either a bigon or annular region containing $f(z)$ and $f^{-1}(\lambda)$
is defined to be $f^{-1}(\lambda')$.  The support $|f^{-1}(\lambda')|$ is contained in $f^{-1}(|\lambda'|)$, and this
containment can be proper since $f^{-1}(|\lambda|)$ may have an isolated leaf.  Note that this happens precisely when
$f(z)$ lies on a boundary leaf of $|\lambda|$.

Train tracks provide a more concrete description of $\widetilde \Psi(f,\lambda)$ which will be useful in proving
continuity results. Let $\mathcal L$ be any geodesic lamination on $S$ and $\epsilon > 0$ sufficiently small so that
the quotient of $N_\epsilon(\mathcal L)$ by collapsing the ties defines a train track $\tau$ as in Section
\ref{S:laminations}. Suppose that $\lambda$ is a measured lamination on $S$ for which $|\lambda|$ is contained in
$N_\epsilon(\mathcal L)$ and is transverse to the ties.  If $f(z) \not \in N_\epsilon(\mathcal L)$, then $\widetilde
\Psi(f,\lambda)$ is the lamination on $(S,z)$ determined by the weighted train track $f^{-1}(\tau(\lambda))$ as
described in Section \ref{S:laminations}.

If $f(z) \in N_\epsilon(\mathcal L)$ then by a small perturbation of $\epsilon$ we may assume that $f(z)$ does not lie on a
boundary-tie of any rectangle and that each switch of $\tau$ is trivalent.  Then either $f(z)$ is outside $N_\epsilon(\mathcal L)$ and we are in the situation above, or else $f(z)$ is in the interior of some rectangle $R$.  Furthermore, $\tau$ can be realized in $N_\epsilon(\mathcal L)$ with the branch $\beta_R$ associated to $R$ contained in $R$.

We modify the train track $\tau$ at the branch $\beta_R$ as follows.  Remove an arc in the interior of $\beta_R$ leaving two subarcs
$\beta_R^\ell$ and $\beta_R^r$ of $\beta_R$.  Insert two branches $\beta_R^u$ and $\beta_R^d$ creating a bigon
containing $f(z)$; see Figure \ref{insertbigon}.  The result, denoted $\tau'$, is a train track on $(S,f(z))$.

If $f_t \in \Diff_0(S)$ is an isotopy with $f = f_0$ and $f_t(z) \subset R$ for every $t \in [0,1]$, and $\tau_t'$ is
constructed for $f_t$ as $\tau$ is constructed for $f$ (so $\tau' = \tau_0'$), then $f_t^{-1}(\tau_t')$ is (isotopic to) $f^{-1}(\tau')$ for all $t$.
\begin{figure}[htb]
\centerline{}
\centerline{}
\begin{center}
\ \psfig{file=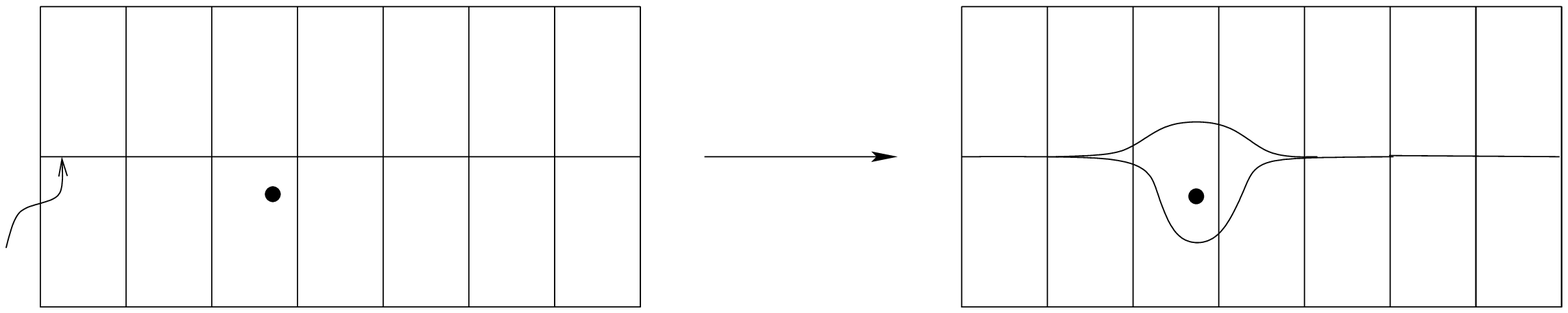,height=.9truein}
\caption{Modifying $\tau$ to $\tau'$.}
\label{insertbigon}
\end{center}
 \setlength{\unitlength}{1in}
  \begin{picture}(0,0)(0,0)
    \put(.72,.8){$f(z)$}
    \put(.02,.65){$\beta_R$}
    \put(.05,1.4){$R$}
    \put(2.97,1.15){$\beta_R^\ell$}
    \put(4.24,1.15){$\beta_R^r$}
    \put(3.5,1.25){$\beta_R^u$}
    \put(3.5,.67){$\beta_R^d$}
  \end{picture}
\end{figure}

The measured lamination $\lambda$ makes $\tau'$ into a weighted train track $\tau'(\lambda)$ on $(S,f(z))$ as follows.
For the branches of $\tau'$ that are the same as those of $\tau$, the weights are defined as before. To define the
weights on the new branches, we first consider the tie $a \subset R$ that contains $f(z)$, and write it as the union of
subarcs $a = a^u \cup a^d$ with $a^u \cap a^d = \{ f(z)\}$.  We define the weights on the branches $\beta_R^u$ and
$\beta_R^d$ of the bigon to be $\lambda(a^u)$ and $\lambda(a^d)$, respectively, while the weights on the branches
$\beta_R^\ell$ and $\beta_R^r$ are both $\lambda(a) = \lambda(a^u)+\lambda(a^d)$; see Figure \ref{bigonweights}. The
lamination $f^{-1}(\lambda)$ is the lamination determined by the weighted train track $f^{-1}(\tau'(\lambda))$.

\begin{figure}[htb]
\centerline{}
\centerline{}
\begin{center}
\ \psfig{file=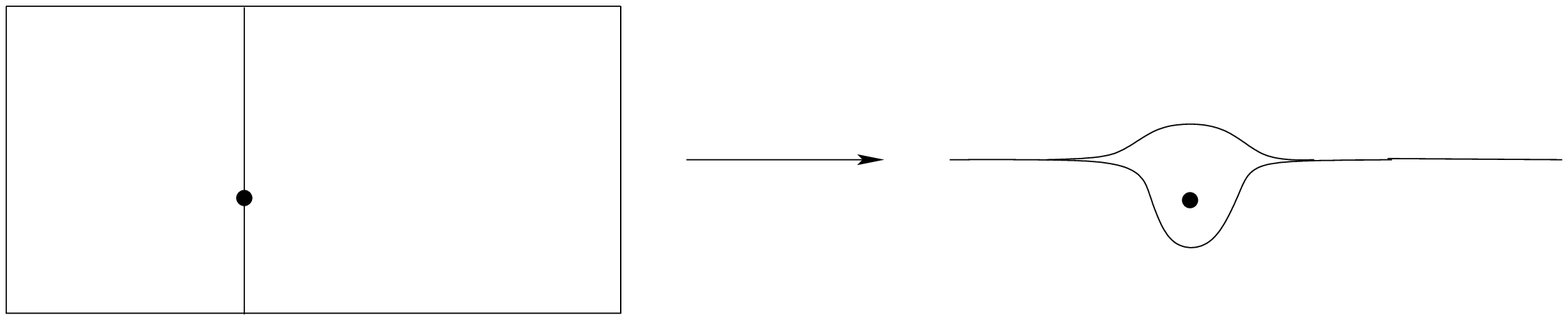,height=.9truein}
\caption{Weights on $\tau'$ determined by $\lambda$ and $f(z)$.}
\label{bigonweights}
\end{center}
 \setlength{\unitlength}{1in}
  \begin{picture}(0,0)(0,0)
    \put(.57,.92){$f(z)$}
    \put(.03,1.4){$R$}
    \put(.9,1.15){$a^u$}
    \put(.9,.75){$a^d$}
    \put(2.9,1.11){$\lambda(a)$}
    \put(4.24,1.11){$\lambda(a)$}
    \put(3.45,1.23){$\lambda(a^u)$}
    \put(3.45,.65){$\lambda(a^d)$}
  \end{picture}
\end{figure}

The proof of the following is similar to that of Proposition \ref{phidescends} and we omit it.
\begin{proposition} \label{psidescends}
$\widetilde \Psi$ descends to a $\pi_1(S)$--equivariant map $\Psi$:
\[
\xymatrix{ \ML(S) \times \Diff_0(S) \ar[drr]^{\widetilde \Psi} \ar[d]_{\id_{\ML(S)} \times \widetilde \ev}\\
 \ML(S) \times \mathbb H \ar[rr]_{\Psi} & & \ML(S,z).}\]
\vspace{-.5cm}

\qed
\end{proposition}

Because of the particular way we have realized our laminations, neither the map $\widetilde \Psi$ nor the map $\Psi$
need be continuous at measured laminations with nontrivial annular component. However, this is the only place where
continuity can break down.

\begin{proposition} \label{tildepsicont}
The map $\widetilde \Psi$ is continuous on $\FL(S) \times \Diff_0(S)$.
\end{proposition}
\begin{proof}  We will show that for any sequence $\{(\lambda_n,f_n)\}$ in $\ML(S) \times \Diff_0(S)$ converging to
$(\lambda,f) \in \FL(S) \times \Diff_0(S)$ there is a subsequence for which $\{\widetilde
\Psi(\lambda_{n_k},f_{n_k})\}$ converges to $\widetilde \Psi(\lambda,f)$.  Since we will find such a subsequence for \textit{any} sequence converging to $(\lambda,f)$, continuity of $\widetilde \Psi$ at $(\lambda,f)$ will follow.

We begin by passing to a subsequence for which
the supports $\{|\lambda_n|\}$ converge in the Hausdorff topology to a closed set $\mathcal L$.
It follows from Lemma \ref{L:convergetolam}, that $\mathcal L$ is a geodesic lamination containing $|\lambda|$.\\

\noindent \textbf{Case 1.} Suppose $f(z) \not \in \mathcal L$.\\

In this case, there is an $\epsilon >0$ so that the $\epsilon$--neighborhoods of $f(z)$ and $\mathcal
L$ are disjoint.  Since $f_n \to f$ as $n \to \infty$, there exists $N >0$ so that for all $n \geq N$, $f_n(z) \in
N_\epsilon(f(z))$, and moreover, $f_n$ is isotopic to $f$ through an isotopy $f_t$ such that $f_t(z) \in
N_\epsilon(f(z))$ for all $t$.  Taking $N$ even larger if necessary, we may assume that for $n \geq N$, $\lambda_n \subset
N_\epsilon(\mathcal L)$.  Therefore, for all $n \geq N$, $\lambda$ and $\lambda_n$ determine weighted train tracks
$\tau(\lambda)$ and $\tau(\lambda_n)$, respectively.  Since $\lambda_n \to \lambda$, it follows that
$\tau(\lambda_n) \to \tau(\lambda)$ as $n \to \infty$.

Since $f_n$ is isotopic to $f$ by an isotopy keeping the image of $z$ in $N_\epsilon(f(z))$, it follows that
$f^{-1}(\tau) = f_n^{-1}(\tau)$, up to isotopy.  Therefore, $f^{-1}(\tau(\lambda_n))$ and $f_n^{-1}(\tau(\lambda_n))$
are isotopic and so we have convergence of weights $f^{-1}(\tau(\lambda_n)) \to f^{-1}(\tau(\lambda))$ which implies
the associated measured laminations converge
\[ \widetilde \Psi(\lambda_n,f_n) \to \widetilde \Psi(\lambda,f)\]
as required.  This completes the proof for Case 1.\\

\noindent \textbf{Case 2.} Suppose that $f(z) \in \mathcal L$.\\

We choose $\epsilon > 0$ sufficiently small so that the quotient of $N_\epsilon(\mathcal L)$ by collapsing ties is a train track $\tau$, so that $f(z)$
lies in the interior of some rectangle $R$ of $N_\epsilon(\mathcal L)$ and so that $\tau$ is trivalent.

Let $N > 0$ be such that for all $n \geq N$, $f_n(z)$ also lies in the interior of $R$ and $f$ is isotopic to $f_n$ by
an isotopy $f_t$ with $f_t(z)$ contained in $R$ for all $t$.  For each $n \geq N$, the train track $\tau$ associated to
$N_\epsilon(\mathcal L)$ and the points $f_n(z)$ and $f(z)$ define tracks $\tau_n'$ and $\tau'$, respectively, with
bigons as described above.  Moreover, $f_n^{-1}(\tau_n')$ and $f^{-1}(\tau')$ are isotopic, and we simply identify the
two as the same train track on $(S,z)$.

Since $\lambda_n$ is converging to $\lambda$ as $n \to \infty$, it follows that the weighted train tracks $\tau(\lambda_n)$
converge to $\tau(\lambda)$.  Therefore, to prove that the weighted train tracks $f^{-1}(\tau_n'(\lambda_n)) =
f_n^{-1}(\tau_n'(\lambda_n))$ converge to $f^{-1}(\tau'(\lambda))$, it suffices to prove that the weights assigned to
$f^{-1}(\beta_R^u)$ and $f^{-1}(\beta_R^d)$ by $\lambda_n$ converge to the weights assigned to these branches by
$\lambda$.  This is sufficient because the weights on the remaining branches agree with weights on the corresponding
branches of $\tau$, where we already know convergence.  From this it will follow that $\widetilde \Psi(\lambda_n,f_n)
\to \widetilde \Psi(\lambda,f)$.

Note that the weights on $\beta_R$ determined by the $\lambda_n$ converge to the weight defined by $\lambda$.  So,
since the sum of the weights on $f^{-1}(\beta_R^u)$ and $f^{-1}(\beta_R^d)$ is precisely the weight on $\beta_R$, it
suffices to prove convergence for the weights of one of these, say, $f^{-1}(\beta_R^u)$.

To define the required weights, first recall that we have the tie $a_n \subset R$ with $f_n(z) \in a_n$, and write
$a_n$ as a union of subarcs $a_n = a_n^u \cup a_n^d$ with $a_n^u \cap a_n^d = \{f_n(z)\}$.  Similarly, we have a tie $a
\subset R$ with $a = a^u \cup a^d$ and $a^u \cap a^d=\{f(z)\}$.  Then the weights on $f^{-1}(\beta_R^u)$ determined by
$\lambda_n$ and $\lambda$ are given by
\[\lambda_n(a_n^u) \quad \mbox{ and } \quad \lambda(a^u),\]
respectively.

Therefore, we must verify that $\lambda_n(a_n^u) \to \lambda(a^u)$.  However, since $T(\lambda_n) \to 0$ as
$k \to \infty$, Lemma \ref{notthatdifferent} implies that it suffices to prove $\hat \lambda_n(a_n^u) \to
\lambda(a^u)$.

Fix any $\delta > 0$. Since $\Cur(\hat \lambda) = \emptyset$, $\hat \lambda|_a$ has no atoms, and so we can find subarcs $a^u_-$ and $a^u_+$ of $a$ with
\[ a^u_- \subsetneq a^u \subsetneq a^u_+ \subset a \]
so that
\[ \hat \lambda(a^u_-) \leq \hat \lambda(a^u) \leq \hat \lambda(a^u_+) \]
with $\hat \lambda(a^u_+) - \hat \lambda(a^u_-) < \delta$.

Since $\hat \lambda_n \to \hat \lambda$, it follows that we also have
\[ \lim_{n \to \infty} \hat \lambda_n(a_+^u) = \hat \lambda(a_+^u)\]
and
\[ \lim_{n \to \infty} \hat \lambda_n(a_-^u) = \hat \lambda(a_-^u).\]

Furthermore, since $a_n \to a$ and $a_n^u \to a^u$ in the $C^1$--topology, we see that
\[ \limsup_{n \to \infty}\hat \lambda_n(a_n^u) \leq \lim_{n \to \infty} \hat \lambda_n(a^u_+) = \hat \lambda(a^u_+)\]
and
\[ \liminf_{n \to \infty} \hat \lambda_n(a_n^u) \geq \lim_{n \to \infty} \hat \lambda_n(a^u_-) = \hat \lambda(a^u_-). \]
Since $\liminf \hat \lambda_n(a_n^u) \leq \limsup \hat \lambda_n(a_n^u)$, combining all of the above, we obtain
\[ \Big| \limsup_{n \to \infty}\hat \lambda_n(a_n^u) - \hat \lambda(a^u) \Big| + \Big| \liminf_{n \to \infty} \hat \lambda_n(a_n^u) - \hat \lambda(a^u) \Big| < \delta. \]
As $\delta$ was arbitrary, it follows that
\[ \lim_{n \to \infty} \hat \lambda_n(a_n^u) = \limsup_{n \to \infty} \hat \lambda_n(a_n^u) = \liminf_{n \to \infty} \hat \lambda_n(a_n^u) = \hat \lambda(a^u)\]
and this completes the proof of Case 2.
Since Cases 1 and 2 exhaust all possibilities, this also completes the proof of the proposition.
\end{proof}

\begin{corollary} \label{psicont}
The map $\Psi$ is continuous on $\FL(S) \times \mathbb H$.
\end{corollary}
\begin{proof}
The map $\widetilde \ev$ is a quotient map.
\end{proof}

%%%%%%%%%%%%%%%%%%%%%%%%%%%%%%
\subsection{$\Phi$ and $\Psi$.}
%%%%%%%%%%%%%%%%%%%%%%%%%%%%%%

The map $\Psi$ descends to a map $\PML(S) \times \mathbb H \to \PML(S,z)$ in the obvious way.  We denote this map by
$\Psi$ with the context deciding the meaning.

We let $\Psi_\C$ denote the restriction of $\Psi$ to $\PML_\C(S) \times \mathbb H$.  The map $\Psi_\C$ has image
$\PML_\C(S,z)$.
\begin{lemma} \label{psiphi}
The following diagram commutes
\[
\xymatrix{
\PML_\C(S) \times \mathbb H \ar[r]^{\Psi_\C} \ar[d] & \PML_\C(S,z) \ar[d]\\
\C(S) \times \mathbb H \ar[r]^{\Phi} & \C(S,z)\\}
\]
The vertical arrows here are the natural maps.
\end{lemma}
\begin{proof} On $\PML_\C(S)$, $\Psi$ was defined in essentially the same way as $\Phi$.
\end{proof}

If we let $\Psi_{\overline \C}$ be the restriction of the map $\Psi$ to $\PML_{\overline \C} \times \mathbb H$, then we have
\begin{proposition} \label{psiphi2}
There is a continuous extension $\hat \Phi: \overline{\C}(S) \times \mathbb H \to \overline{\C}(S,z)$ which fits
into a commutative diagram
\[
\xymatrix{
\PML_{\overline \C}(S) \times \mathbb H \ar[r]^{\Psi_{\overline \C}} \ar[d] & \PML_{\overline \C}(S,z) \ar[d]\\
\overline \C (S) \times \mathbb H \ar[r]^{\hat \Phi} & \overline{\C}(S,z)}
\]
\end{proposition}
\begin{proof}  Via Klarreich's work, as discussed in Section \ref{mlandcc}, we identify $\partial \C$ with
$\EL$.  Moreover, the vertical maps in the statement of the proposition send $\mathbb P \FL(S) \times \mathbb H$ and
$\mathbb P \FL(S,z)$ onto $\EL(S) \times \mathbb H$ and $\EL(S,z)$, respectively, using this identification.

From the construction of $\Psi$ and the definition of $\FL$, one can see that
\[ \Psi(\FL(S) \times \mathbb H) \subset \FL(S,z).\]
Furthermore, if $\lambda,\lambda' \in \ML(S)$ with $|\lambda| = |\lambda'|$, then $|\Psi(\lambda,x)| =
|\Psi(\lambda',x)|$.  Thus, $\Psi$ determines a map
\[ \EL(S) \times \mathbb H \to \EL(S,z)\]
which extends $\Phi$ to the required map
\[ \hat \Phi: \overline \C(S) \times \mathbb H \to \overline\C(S,z).\]
Continuity follows from Proposition \ref{klarreich} and Corollary \ref{psicont}.
\end{proof}

We will also need the following
\begin{proposition} \label{contfailure}
Suppose $\{v_n\} \subset \C(S)$, $\{x_n \} \subset \mathbb H$, and $x_n \to x \in \mathbb H$.  If $\{v_n\}$ does not
accumulate on $\partial \C(S)$, then $\{\Phi(v_n,x_n)\}$ does not accumulate on $\partial \C(S,z)$.
\end{proposition}
\begin{proof}
The proof is by contradiction.  Suppose $\{\Phi(v_n,x_n)\}$ accumulates on some lamination $|\mu| \in \partial
\C(S,z)$, and pass to a subsequence which converges to $|\mu|$ in $\overline \C(S,z)$.  If any curve in the sequence
$\{v_n\}$ occurs infinitely often, then passing to a further subsequence, we can assume $v_n$ is constant and equal to
$v$.  Then
\[ |\mu| = \lim_{n \to \infty} \Phi(v_n,x_n) = \lim_{n \to \infty} \Phi(v,x_n) = \Phi(v,x) \in \C(S,z).\]
This is a contradiction since $|\mu| \in \partial \C(S,z)$.  So without loss of generality, we may assume that all the
$v_n$ are distinct.

Fix elements $\lambda_n \in \ML(S)$ representing the projective classes associated to $v_n$ via the natural bijection
$\PML_\C(S) \to \C(S)$.  After passing to a further subsequence if necessary, we may assume that we have convergence
$\lambda_n \to \lambda$.  Since $v_n$ are all distinct, $T(\lambda_n) \to 0$.  Thus, as in the proof of Lemma
\ref{L:convergetolam}, we may pass to a further subsequence if necessary so that $|\lambda_n|$ converges to a geodesic
lamination $\mathcal L$.

It follows from Proposition \ref{klarreich} that no sublamination of $\mathcal L$ lies in $\EL(S)$. In particular,
removing the infinite isolated leaves of $\mathcal L$, we obtain a lamination which is disjoint from a simple closed
curve $v'$ and contains the support of $\hat \lambda$. Choosing $\epsilon > 0$ sufficiently small, we can assume that
the train track $\tau$ obtained from $N_\epsilon(\mathcal L)$ (as described in Section \ref{S:laminations}) contains a
subtrack $\tau_0$ so that (1) $\tau_0$ is disjoint from some representative $\alpha$ of $v'$ and (2) $\tau(\lambda)$
has nonzero weights only on the branches of $\tau_0$.

Now let $f \in \Diff_0(S)$ be such that $\widetilde{\ev}(f) = x$.  After modifying $\tau$ and $\tau_0$ to $\tau'$ and
$\tau_0'$ as in the previous section if necessary (that is, possibly inserting a bigon around $f(z)$), it follows that
for sufficiently large $n$, $f^{-1}(\tau'(\lambda_n))$ determines the lamination $\Psi(\lambda_n,x_n)$. After passing
to yet a further subsequence if necessary, we can assume that $f^{-1}(\tau'(\lambda_n))$ converges to some
$f^{-1}(\tau')(\mu_0)$, also having nonzero weights only on $f^{-1}(\tau_0')$.  It follows that $\mu_0$, the limit of
$\Psi(\lambda_n,x_n)$, is \textit{not} in $\FL(S,z)$ since its support is disjoint from $f^{-1}(\alpha)$. Since
$\mathbb P \Psi(\lambda_n,x_n) \in \PML_\C(S,z)$, Proposition \ref{klarreich} implies $|\mu_0| = |\mu|$, which is a
contradiction.
\end{proof}

\begin{lemma} \label{injective comp reg}
For $(|\lambda|,x),(|\lambda'|,x') \in \EL(S) \times \mathbb H$, $\hat \Phi(|\lambda|,x) = \hat \Phi(|\lambda'|,x')$ if
and only if $|\lambda| = |\lambda'|$ and $x$ and $x'$ are in the same leaf of $p^{-1}(|\lambda|)$ or in the closure of
the same complementary region of $\mathbb H - p^{-1}(|\lambda|)$.
\end{lemma}
\begin{proof}
If $x,x'$ lie on the same leaf of $p^{-1}(|\lambda|)$ or in the closure of the same component of $\mathbb H -
p^{-1}(|\lambda|)$, then it is straightforward to see that $\hat\Phi(|\lambda|,x) = \hat\Phi(|\lambda|,x')$.

Now we prove the forward direction; suppose that $\hat \Phi(|\lambda|,x) = \hat \Phi(|\lambda'|,x')$. We must show that
$|\lambda| = |\lambda'|$ and $x$ and $x'$ are in the same leaf of $p^{-1}(|\lambda|)$ or in the closure of the same
complementary region of $\mathbb H - p^{-1}(|\lambda|)$.

We first apply an isotopy so that the laminations $\hat \Phi(|\lambda|,x)$ and $\hat \Phi(|\lambda'|,x')$ are equal
(not just isotopic). Forgetting $z$, the laminations remain the same (though they may have a bigon complementary
region, and so are not necessarily geodesic laminations), and hence $|\lambda| = |\lambda'|$.

Proving the statement about $x$ and $x'$ is slightly more subtle.  For simplicity, we assume that $x$ and $x'$ lie in
components of $\mathbb H - p^{-1}(|\lambda|)$ (the general case is similar, but the notation is more complicated). Let
$f,f' \in \Diff_0(S)$ be such that $\widetilde \ev(f) = x$ and $\widetilde \ev(f') = x'$.  Let $\widetilde f$ and
$\widetilde f'$ be lifts of $f$ and $f'$ with $\widetilde f(\widetilde z) = x$ and $\widetilde f'(\widetilde z) = x'$
(see Section \ref{S:bundle over H}). Modifying $f$ and $f'$ by an element of $\Diff_0(S,z)$ if necessary, we may assume
that $f^{-1}(|\lambda|) = \hat \Phi(|\lambda|,x)$ and $f'^{-1}(|\lambda|) = \hat \Phi(|\lambda|,x')$ are equal (again,
not just isotopic).

Since $f^{-1}(|\lambda|) = f'^{-1}(|\lambda|)$, it follows that $f' \circ f^{-1}(|\lambda|) = |\lambda|$.  Back in
$\mathbb H$ this means $\widetilde f' \circ \widetilde f^{-1}(p^{-1}(|\lambda|)) = p^{-1}(|\lambda|)$. Since
$\widetilde f' \circ \widetilde f^{-1}(x) = x'$, and $\widetilde f' \circ \widetilde f^{-1}$ is the identity on
$\partial \mathbb H$, it must be that $x$ and $x'$ lie in the same complementary region of $\mathbb H -
p^{-1}(|\lambda|)$, as required.
\end{proof}

%For any $\lambda \in \PFL(S)$ and $x \in \mathbb H$, suppose $f \in \Diff_0(S)$ with $\widetilde \ev(f) = x$.  Then provided $f(z) \not \in |\lambda|$, $\Psi(\lambda,x) = f^{-1}(\lambda)$ and so all complementary regions of $|\Psi(\lambda,x)|$ are either polygons or a polygon containing $z$, hence $\Psi(\lambda,x) \in \PFL(S,z)$.  If $f(z) \in |\lambda|$, then the complementary regions of $|\Psi(\lambda,x)|$ are polygons and a single bigon contain $z$.  We also see that, if $|\lambda| = |\lambda'|$, and $f(z) \not \in |\lambda|$ then $f(z) \not \in |\lambda'|$ and $|\Psi(\lambda,x)| = f^{-1}(|\lambda|) = f^{-1}(|\lambda'|) = |\Psi(\lambda',x)|$.  If $f(z) \in |\lambda| = |\lambda'|$, then $|\Psi(\lambda,x)|$ and $|\Psi(\lambda',x)|$ are preimages of the same laminations with bigons inserted in the same places, so $|\Psi(\lambda,x)| = |\Psi(\lambda',x)|$.

%Recall from Section 1.2.4 that $\PML_{\overline \C} = \PML_\C \bigcup \PFL$ and that the natural surjections from $\PML_{\overline \C}$ to ${\overline \C}$ are continuous at every point of $\PFL$ by Proposition \ref{klarreich}.  It follows that

% The existence of a continuous $\widetilde{\Phi}$ fitting into the above commutative diagram now follows from Corollary \ref{psicont} which ensures the continuity of ${\Psi_{\overline \C}}$ at the points of $\partial \C (S) = \overline \C (S) - \C$.
%\end{proof}

%%%%%%%%%%%%%%%%%%%%%%%%%%%%%%%%%%%%%%%%%%%%%%%%
\section{Universal Cannon--Thurston maps.}
%%%%%%%%%%%%%%%%%%%%%%%%%%%%%%%%%%%%%%%%%%%%%%%

%%%%%%%%%%%%%%%%%%%%%%%%%%%%%%%%%%%%
\subsection{Quasiconvex sets.}
%%%%%%%%%%%%%%%%%%%%%%%%%%%%%%%%%%%%

For the remainder of the paper, fix a biinfinite geodesic $\gamma \subset \mathbb H$ for which $p(\gamma)$ is a
\textit{filling} closed geodesic in $S$, by which we mean that $p(\gamma)$ is a closed geodesic and the complement of
$p(\gamma)$ is a union of disks in $S$.  Let $\delta \in \pi_1(S)$ generate the (infinite cyclic) stabilizer of
$\gamma$. We will make several statements about $\gamma$, though they will also obviously apply to any
$\pi_1(S)$--translate of $\gamma$.

Define
\[ \X(\gamma) = \Phi(\C(S) \times \gamma). \]
Let $H^\pm(\gamma)$ denote the two half spaces bounded by $\gamma$ and define
\[ \H^{\pm}(\gamma) = \Phi(\C(S) \times H^\pm(\gamma)). \]

Recall that $N(v) = N_{\epsilon(v)}(v)$ is a small neighborhood of the geodesic representative of $v \in \C^0(S)$.  We
may assume that the $\epsilon(v)$ are small enough to ensure that every component $\alpha \subset \gamma \cap
p^{-1}(N(v))$ is essential in the strip of $p^{-1}(N(v))$ that $\alpha$ meets.  Here, we say that arc is
\textit{essential} if it is not homotopic into the boundary keeping the endpoints fixed.

A subset $X$ of a geodesic metric space is called \textit{weakly convex} if for any two points of the set there exists
a geodesic connecting the points contained in the set.  In a Gromov hyperbolic space, weakly convex sets are in
particular uniformly quasi-convex.

\begin{proposition} \label{weakconvex} $\X(\gamma),\H^\pm(\gamma)$ are simplicial subcomplexes of $\C(S,z)$ spanned by their vertex sets and are weakly
convex.
\end{proposition}
To say that a subcomplex $\Omega \subset \C(S,z)$ is \textit{spanned by its vertex set}, we mean that $\Omega$ is the
largest subcomplex having $\Omega^{(0)}$ as its vertex set.

\begin{proof}[Proof of Proposition \ref{weakconvex}.] We describe the case of $\X(\gamma)$, with $\H^\pm(\gamma)$ handled by similar arguments.
First we appeal to Proposition \ref{prop: factors through trees} and Theorem \ref{glued cc} to describe the structure
of $\X(\gamma) \subset \C(S,z)$.  Next we prove that $\X(\gamma)$ is spanned by its vertices and finally we construct a
simplicial projection $\rho : \C(S,z) \to \X(\gamma)$. The existence of $\rho$ implies the proposition.

For any $x \in \INT(v)$, $\X(\gamma) \cap \Pi^{-1}(x) = \Phi(\{x\} \times \gamma)$, which is a biinfinite geodesic in
the tree $\Pi^{-1}(x) \cong T_v$.  One can also see this as the axis of $\delta$ in $T_v$ (since $p(\gamma)$ is
filling, $\delta$ is not elliptic in $T_v$); we denote this axis by $\gamma_v \subset T_v$. Recall that an inclusion of
faces $\phi:v' \to v$ induces a quotient of associated trees $\phi^*:T_v \to T_{v'}$.  Since the axis of $\delta$ in
$T_v$ is sent to the axis of $\delta$ in $T_{v'}$ by $\phi^*$, we have $\phi^*(\gamma_v) = \gamma_{v'}$.  Therefore,
with respect to our homeomorphism with the quotient of Theorem \ref{glued cc}, we have
\begin{equation} \label{glued Xgamma}
\X(\gamma) \cong \left. \left( \bigsqcup_{v \subset \C(S)} v \times \gamma_v \right) \right/ \sim \end{equation}
where, as in Theorem \ref{glued cc}, the disjoint union is over all simplices $v \subset \C(S)$, and the equivalence relation is generated by
\[(\phi(x),t) \sim (x,\phi^*(t)) \]
for all faces $\phi:v' \to v$, all $x \in v'$ and all $t \in \gamma_v$.
We also use the homeomorphism in \eqref{glued Xgamma} to identify the two spaces.

We can now show that $\X(\gamma)$ is spanned by its vertices.
The simplices of $\C(S,z)$ via the homeomorphism of Theorem \ref{glued cc} are precisely the images of cells $v \times
\sigma$ in the quotient, where $v \subset \C(S)$ is a simplex and $\sigma \subset T_v$ is an edge or vertex.   Thus, if
the image of $v \times \sigma$ is a simplex, and we let $v_0,...,v_k$ be the vertices of $v$ and $t_0,t_1$ the vertices
of $\sigma$ (assuming, for example, that $\sigma$ is an edge) then the vertices of the simplex determined by $v \times
\sigma$ are images of $(v_i,t_j)$ for $i =0,...,k$ and $j = 0,1$.  If these vertices lie in $\X(\gamma)$, then $t_0,t_1
\in \gamma_v$, hence $\sigma \subset \gamma_v$ and the image of $v \times \sigma$ lies in $\X(\gamma)$. It follows that
$\X(\gamma)$ is a simplicial subcomplex of $\C(S,z)$ spanned by its vertex set.

Next, we will define a projection
\[\rho:\C(S,z) \to \X(\gamma).\]
Let $\eta_v : T_v \to \gamma_v$ be the nearest point projection
map.  Extend $\eta_v$ by the identity map to obtain $\rho_v : v
\times T_v \to v \times \gamma_v$.
Observe that if $\phi:v' \to v$ is a face, then nearest-point projections commute
\[ \eta_{v'} \circ \phi^* = \phi^* \circ \eta_v.\]
This is because a geodesic segment in $T_v$ from a point $t$ to $\gamma_v$ is taken to a geodesic segment from $\phi^*(t)$ to $\gamma_{v'}$.
From this it follows that the maps $\rho_v$ give a well-defined map $\rho$.

All that remains is to verify that $\rho$ is simplicial.  Given a simplex which is the image of $v \times \sigma$ in
the quotient, for some $v \subset \C(S)$ and $\sigma \subset T_v$, the $\rho$--image of this simplex is the image of
$\rho_v(v \times \sigma) = v \times \eta_v(\sigma)$ in the quotient.  Since $\eta_v(\sigma)$ is either an edge or
vertex, $v \times \eta_v(\sigma)$ projects to a simplex in the quotient, as required.
\end{proof}

Throughout what follows we continue to denote the axis of $\delta$ in $T_v$ by $\gamma_v \subset T_v$ or, with respect to the homeomorphism $T_v \cong \Pi^{-1}(v)$, by $\gamma_v = \Phi(\{v\} \times \gamma)$.

\begin{proposition} \label{separating half spaces}
We have
\[ \H^+(\gamma) \cup \H^-(\gamma) = \C(S,z) \]
and
\[ \H^+(\gamma) \cap \H^-(\gamma) = \X(\gamma).\]
\end{proposition}
\begin{proof}  The first statement follows from the fact that $H^+(\gamma) \cup H^-(\gamma) =
\mathbb H$ and that $\Phi$ is surjective.

For the second statement, first observe that since $\gamma \subset H^\pm(\gamma)$, it follows that
\[ \X(\gamma) \subset \H^+(\gamma) \cap \H^-(\gamma).\]
To prove the other inclusion, look in each of the trees $\Pi^{-1}(v) \cong T_v$.
For each vertex $v \in \C(S)$, we define the \textit{half-tree}
\[ H^\pm(\gamma_v) := \H^\pm(\gamma) \cap \Pi^{-1}(v) = \Phi(\{v\} \times H^+(\gamma)).\]

Let $u \in H^+(\gamma_v) \cap H^-(\gamma_v)$ be any vertex; we will show that $u \in \gamma_v$.
We can write $u = \Phi(\{v\} \times U)$ where $U$ is a component of $\mathbb H - p^{-1}(N(v))$.
Therefore, $U \cap H^+(\gamma) \neq \emptyset$ and $U \cap H^-(\gamma) \neq \emptyset$.  Since $U$ is connected and
$\gamma$ separates $H^+(\gamma)$ from $H^-(\gamma)$ we have $U \cap \gamma \neq \emptyset$.  Hence $\Phi(\{v\} \times
U) = u \in \gamma_v$ as required.

Given any simplex $u = \{u_0,...,u_k\} \subset \H^+(\gamma) \cap \H^-(\gamma)$, by the previous paragraph we have $u_j \in \X(\gamma)$.  Since $\X(\gamma)$ is a subcomplex spanned by its vertex set, we have $u \subset \X(\gamma)$ and hence
\[\H^+(\gamma) \cap \H^-(\gamma) \subset \X(\gamma)\]
which completes the proof.
\end{proof}

It will be convenient to keep the terminology in the proof of this proposition as well. We therefore think of $\gamma_v$ as ``bounding the half-trees'' $H^\pm(\gamma_v) \subset T_v \cong \Pi^{-1}(v)$.

%%%%%%%%%%%%%%%%%%%%%%%%%%%%%%%%%%
\subsection{Rays and existence of Cannon--Thurston maps.} \label{S:rays and existence of ct maps}
%%%%%%%%%%%%%%%%%%%%%%%%%%%%%%%%%%

An \textit{essential subsurface} of $S$ is either a component of the complement of a geodesic multicurve in $S$, the annular neighborhood $N(v)$ of some geodesic $v \in \C^0(S)$, or else the entire surface $S$.

A point $x \in \partial \mathbb H$ is a {\em filling point} for an essential subsurface $Y$ (or simply,
$x$ {\em fills} $Y$) if
\begin{itemize}
\item
for every geodesic ray $r \subset \mathbb H$ ending at $x$ and for every $v
\in \C^0(S)$ which nontrivially intersects $Y$, we have $p(r) \cap v \neq \emptyset$ and
\item
there is a geodesic ray $r \subset \mathbb H$ ending at $x$ so that $p(r)
\subset Y$.
\end{itemize}

Observe, by taking subrays, that any such ray $r$ in fact meets
infinitely many components of $p^{-1}(v)$.  Observe also that every
point $x \in \partial \mathbb H$ fills exactly one essential subsurface of $S$.

Let $\ATF \subset \partial \mathbb H$ be the set of points that fill $S$.
\begin{lemma}
If $x \not \in \ATF$ and $r$ is a ray ending at $x$ then $\Phi(\{v\} \times
r)$ has bounded diameter for all $v \in \C^0(S)$.
\end{lemma}

\begin{proof}
Since $x$ does not fill $S$ there is a simple closed geodesic $v' \subset S$ so that $p(r) \cap v'$ is finite.  It follows that $\Phi(\{v'\} \times
r)$ has bounded diameter in $\Pi^{-1}(v') \subset \C(S, z)$.  Since
$\Phi(\{v'\} \times r)$ and $\Phi(\{v\} \times r)$ have bounded Hausdorff distance, we are done.
\end{proof}

Recall that we have fixed once and for all a geodesic $\gamma \subset \mathbb H$ which projects to a non-simple closed
geodesic in $S$.  Consider a set $\{\gamma_n\}$ of pairwise distinct $\pi_1(S)$--translates of $\gamma$, with the
property that the half spaces are nested:
\[ H^+(\gamma_1) \supset H^+(\gamma_2) \supset ... \]
Since the $\gamma_n$ are all distinct, proper discontinuity of the action of $\pi_1(S)$ on $\mathbb H$ implies that
\[ \bigcap_{n=1}^\infty \overline{H^+(\gamma_n)} = \{x\} \]
for some $x \in \partial \mathbb H$.  Here the bar denotes closure in $\overline{\mathbb H} = \mathbb H \cup
\partial \mathbb H$.  We say that $\{\gamma_n\}$ \textit{nests down to $x$}.  Note that $\{ \overline{H^+(\gamma_n)}\}$
is a neighborhood basis for $x$.

Given any $x \in \partial \mathbb H$, if $r \subset \mathbb H$ is a geodesic ray ending at $x$, then since $p(\gamma)$
is filling, $p(r)$ intersects $p(\gamma)$ infinitely often.  It follows that there is a sequence $\{\gamma_n\}$ which
nest down to $x$.

\begin{proposition} \label{nesting distance}
If $\{\gamma_n\}$ is a sequence nesting down to a point $x \in \ATF$, then for any choice of basepoint $u_0 \in \C(S,z)$,
\[d(u_0,\H^+(\gamma_n)) \to \infty\]
as $n \to \infty$.
\end{proposition}
\begin{proof}
Recall that the curve complex and its one-skeleton are quasi-isometric \cite{masur-minsky}.
Thus, in what follows all distances will be computed in the $1$--skeleton.
We write $u_0 = \Phi(v_0,y)$ for some vertex $v_0 \in \C(S)$ and $y \in \mathbb H$.  By discarding a finite number of initial elements of the sequence $\{\gamma_n\}$ we may assume that $y \in H^-(\gamma_n)$ for all $n$, and so $u_0 \in \H^-(\gamma_n)$ for all $n$.

Now, fix any $R>0$.  Since
\[ \H^+(\gamma_1) \supset \H^+(\gamma_2) \supset \H^+(\gamma_3) \supset \ldots \]
we must show that there exists $N > 0$ so that for all $u \in \H^+(\gamma_N)$, $d(u_0,u) \geq R$.\\

\noindent {\bf Claim 1.}
It suffices to prove that there exists $N > 0$, so that for all $u \in (\H^+(\gamma_N) \cap \Pi^{-1}(B(v_0,R)))$, the distance inside $\Pi^{-1}(B(v_0,R))$ from $u_0$ to $u$ is at least $R$.\\
\begin{proof}
Observe that any edge path from a point $u \in \C(S,z)$ to $u_0$ which meets $\C(S,z) - \Pi^{-1}(B(v_0,R))$ projects to
a path which meets both $\C(S) - B(v_0,R)$ and $v_0$, and therefore has length at length at least $R$.  Since $\Pi$ is
simplicial, the length of the path in $\C(S,z)$ is also at least $R$.
\end{proof}

The intersection of $\H^+(\gamma_n)$ with each fiber $\Pi^{-1}(v) \cong T_v$ is a half-tree denoted by $H^+(\gamma_{n,v})$
and bounded by $\gamma_{n,v} = \X(\gamma_n) \cap \Pi^{-1}(v)$.  See the proof of Proposition \ref{separating half
spaces} and comments following it.\\

\noindent {\bf Claim 2.}
For any $k > 0$, there exists positive integers $N_1 < N_2 < N_3 < ...< N_k$ so that
\begin{equation} \label{well spaced eq}
\gamma_{N_j,v} \cap \gamma_{N_{j+1},v} = \emptyset
\end{equation}
for all $j = 1,...,k-1$ and all $v \in B(v_0,R)$.\\
\begin{proof}
The proof is by induction on $k$.  For $k = 1$, the condition is vacuously satisfied by setting $N_1 = 1$.  So, we assume it is true for $k \geq 1$, and prove it true for $k+1$. Thus, by hypothesis, we have found $N_1 < N_2 < ... < N_k$ so that \eqref{well spaced eq} is true, and we need to find $N_{k+1}$ so that
\begin{equation} \label{well spaced eq2}
\gamma_{N_k,v} \cap \gamma_{N_{k+1},v} = \emptyset
\end{equation}
for all $v \in B(v_0,R)$.

We suppose that no such $N_{k+1}$ exists and arrive at a contradiction.  Observe that the nesting
\[H^+(\gamma_{1,v}) \supset H^+(\gamma_{2,v}) \supset ... \]
means that if $\gamma_{n,v} \cap \gamma_{m,v} = \emptyset$ for some $m > n$, then $\gamma_{n,v} \cap \gamma_{m+j,v} = \emptyset$ for all $j \geq 0$.

Thus, since no such $N_{k+1}$ exists, it must be that for every $j > 0$, there exists a curve $v_j \in B(v_0,R)$ so that
\[ \gamma_{N_k,v_j} \cap \gamma_{N_k+j,v_j} \neq \emptyset.\]
Let $u_j \in \gamma_{N_k,v_j} \cap \gamma_{N_k+j,v_j}$ be a vertex in the intersection.  This vertex is the image of a component $U_j \subset \mathbb H - p^{-1}(v_j)$ which meets both $\gamma_{N_k}$ and $\gamma_{N_k+j}$; see Figure \ref{wellspacedfiglabel}.
\begin{figure}[htb]
\begin{center}
\ \psfig{file=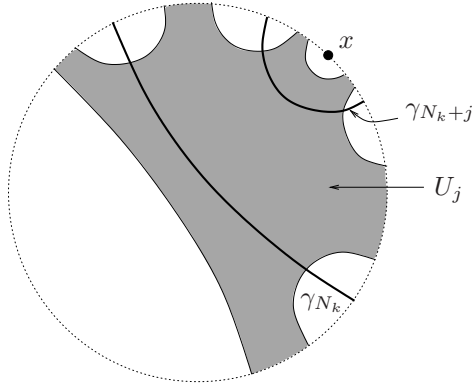,height=2truein}
\caption{The region $U_j$ and the geodesics $\gamma_{N_k}$ and $\gamma_{N_k+j}$ from the sequence nesting down on $x$.}
\label{wellspacedfiglabel}
\end{center}
 \setlength{\unitlength}{1in}
  \begin{picture}(0,0)(0,0)
    \put(2.85,1.15){$\gamma_{N_k}$}
    \put(3.4,2.15){$\gamma_{N_k+j}$}
    \put(3.55,1.73){$U_j$}
    \put(3.05,2.5){$x$}
  \end{picture}
\end{figure}

Let $g_j \subset U_j \subset \mathbb H$ be a geodesic segment connecting a point $y_j^- \in \gamma_{N_k}$ to $y_j^+ \in \gamma_{N_k+j}$.
Furthermore, we may pass to a subsequence so that the $y_j^-$
converge to some point $y$ (possibly in $\partial \mathbb H$) of
$\gamma_{N_k}$. It follows that the sequence of geodesics $g_j$
converge to a geodesic ray or line $r_\infty$ connecting $y$ and $x$.

By passing to a further subsequence, we can assume that $v_j$ limits in the Hausdorff topology to a geodesic lamination
$\mathcal L$, and that $p(r_\infty)$ does not transversely intersect $\mathcal L$.  Because the $v_j$ are all contained
in $B(v_0,R)$, $\mathcal L$ cannot contain an ending lamination as a sublamination by Proposition \ref{klarreich}. It
follows from \cite{CB} that $\mathcal L$ is obtained from a lamination supported on a proper subsurface $\Sigma$ by
adding a finite number of isolated leaves.  Any geodesic in $S$ which does not transversely intersect $\mathcal L$ can
only transversely intersect $\partial \Sigma$ twice (when it possibly exits/enters a crown; see \cite{CB}). Since
$p(r_\infty)$ meets $\partial \Sigma$ at most twice the point $x$ does not fill $S$, a contradiction.
\end{proof}

Now, pick an integer $k > R + 1$ and let $N_1 < N_2 < ... < N_k$ be as in Claim 2.  There can be no vertices in $\X(\gamma_{N_j}) \cap \X(\gamma_{N_{j+1}}) \cap \Pi^{-1}(B(v_0,R))$, and hence
\[ \X(\gamma_{N_j}) \cap \X(\gamma_{N_{j+1}}) \cap \Pi^{-1}(B(v_0,R)) = \emptyset. \]
Moreover, since
\[ \H^+(\gamma_{N_1}) \supset \H^+(\gamma_{N_2}) \supset ... \supset \H^+(\gamma_{N_k}) \]
it follows from Proposition \ref{separating half spaces} that
\begin{equation} \label{completely disjoint}
\X(\gamma_{N_j}) \cap \X(\gamma_{N_i}) \cap \Pi^{-1}(B(v_0,R)) = \emptyset
\end{equation}
for all $i \neq j$ between $1$ and $k$.

Let $u \in \H^+(\gamma_{N_k}) \cap \Pi^{-1}(B(v_0,R))$ be any point and $\{u_0,u_1,...,u_m=u\}$ be the vertices of an edge path from $u_0$ to $u$ within $\Pi^{-1}(B(v_0,R))$.  We have $u_0 \in \H^-(\gamma_{N_j})$ for all $j$ and $u \in \H^+(\gamma_{N_k}) \subset \H^+(\gamma_{N_j})$ for all $j$.  By Proposition \ref{separating half spaces}, the edge path must meet $\X(\gamma_{N_j})$ for each $j$.  That is, for each $j$, there is some $i = i(j)$ so that $u_{i(j)} \in \X(\gamma_{N_j})$.  By \eqref{completely disjoint}, there must therefore be at least $k > R+1$ vertices in the path, and hence the length of the path is at least $R$.

Therefore, setting $N = N_k$, we have for all $u \in \H^+(\gamma_N) \cap \Pi^{-1}(B(v_0,R))$, the distance inside $\Pi^{-1}(B(v_0,R))$ from $u_0$ to $u$ is at least $R$.  By Claim 1, this completes the proof of the proposition.
\end{proof}

We can now prove the first half of Theorem \ref{cannonthurston}.

\begin{theorem} \label{cannonthurstonhalf}
For any $v \in \C^0(S)$, the map
$$\Phi_v: \mathbb H \to \C(S,z)$$
has a continuous $\pi_1(S)$--equivariant extension to
$$\overline \Phi_v: \mathbb H \cup \ATF \to \overline \C(S,z).$$
\end{theorem}

\begin{proof}  Observe that $\Phi_v$ is already defined and continuous.  All that remains is to extend it to $\overline
\Phi_v$ on $\ATF$ by checking the criterion of Lemma \ref{ct-crit}.  Equivariance will follow from equivariance of $\Phi_v$, continuity of $\overline \Phi_v$, and the fact that $\mathbb H$ is dense in $\mathbb H \cup \ATF$.

Fix a basepoint $u_0 \in \C(S,z)$. Given any $x \in \ATF$, let $\{\gamma_n\}$ be any sequence nesting down on $x$.
According to Proposition \ref{nesting distance}, we have
\[ d(u_0,\H^+(\gamma_n)) \to \infty. \]
Moreover, by Proposition \ref{weakconvex}, $\H^+(\gamma_n)$ is weakly convex and hence uniformly quasi-convex. Finally,
observe that $\Phi_v(H^+(\gamma_n)) = \Phi(\{v\} \times H^+(\gamma_n)) \subset \H^+(\gamma_n)$.  Since $x \in \ATF$ was
an arbitrary point, Lemma \ref{ct-crit} implies the existence of an $\ATF$--Cannon-Thurston map $\overline\Phi_v$.
\end{proof}

We note that, given $x \in \ATF$, the image $\overline\Phi_v(x)$ depends only on $x$, not on $v$, and is the
unique point of intersection of the sets
\[\bigcap_n \overline{\H^+(\gamma_n)}. \]
We can therefore unambiguously define $\partial \Phi: \ATF \to \partial \C(S,z)$ by $\partial \Phi(x) =
\overline\Phi_v(x)$ for any $x \in \ATF$, independent of the choice of $v \in \C^0(S)$.

%%%%%%%%%%%%%%%%%%%%%%%%%%%%%%%%%%%%%%
\subsection{Separation.}
%%%%%%%%%%%%%%%%%%%%%%%%%%%%%%%%%%%%%%

\begin{proposition} \label{disjointhoods}
Given $x,y \in \ATF$, let $\epsilon$ be the geodesic connecting them.  Then there are
$\pi_1(S)$--translates $\gamma_x$ and $\gamma_y$ of $\gamma$ defining half-space neighborhoods $\overline{H^+(\gamma_x)}$ and
$\overline{H^+(\gamma_y)}$ of $x$ and $y$, respectively, with
\[\partial \H^+(\gamma_x) \cap \partial \H^+(\gamma_y) = \emptyset\]
if and only if $p(\epsilon)$ is non-simple.
\end{proposition}

Before we can give the proof of Proposition \ref{disjointhoods}, we will need the analogue of Proposition
\ref{separating half spaces} for the boundaries at infinity.  Recall that $\gamma$ was chosen to be a  biinfinite
geodesic with stabilizer $\langle \delta \rangle$ and $p(\gamma)$ a filling closed geodesic.

\begin{proposition} \label{sep half infinity}We have
\[\partial \H^+(\gamma) \cup \partial \H^-(\gamma) = \partial \C(S,z)\]
and
\[\partial \H^+(\gamma) \cap \partial \H^-(\gamma) = \partial \X(\gamma).\]
\end{proposition}
\begin{proof}
This first statement in an immediate consequence of Proposition \ref{separating half spaces}.
The second also follows from this proposition, but requires some additional argument.
Since $\X(\gamma) = \H^+(\gamma) \cap \H^-(\gamma)$, it easily follows that
\[ \partial \X(\gamma) \subset \partial \H^+(\gamma) \cap \H^-(\gamma).\]

If $|\mu| \in \partial \H^+(\gamma) \cap  \partial \H^-(\gamma)$, then let $\{u_n^+\} \subset \H^+(\gamma)$ and
$\{u_n^-\} \in \H^-(\gamma)$ be sequences converging to $|\mu|$ in $\C(S,z)$. Let $g_n$ be geodesic segments from
$u_n^+$ to $u_n^-$.  By Proposition \ref{separating half spaces}, there is a vertex $u_n \in g_n \cap \X(\gamma)$.
Therefore $u_n$ also converges to $|\mu|$, and so $|\mu| \in
\partial \X(\gamma)$, proving
\[ \partial \H^+(\gamma) \cap \partial \H^-(\gamma) \subset \partial \X(\gamma).\]
\end{proof}

A theorem of Kra \cite{kra} implies that, since $p(\gamma)$ is filling on $S$, $\delta$ is pseudo-Anosov as an element
of $\Mod(S,z)$.  We let $|\mu_+|$ and $|\mu_-|$ be the attracting and repelling fixed points of $\delta$, respectively,
in $\partial \C(S,z)$.
\begin{lemma} \label{boundary Xgamma}
\[\partial \X(\gamma) = \hat\Phi(\partial \C(S) \times \gamma) \cup \{|\mu_\pm|\}\]
\end{lemma}
\begin{proof}
Continuity of $\hat \Phi$ implies $\hat\Phi(\partial \C(S) \times \gamma) \subset \partial \X(\gamma)$. Invariance of
$\gamma$ by $\delta$ implies invariance of $\X(\gamma)$ by $\delta$ so $\{ |\mu_\pm|\} \subset
\partial \X(\gamma)$, and hence
\[ \partial \X(\gamma) \supset \hat\Phi(\partial \C(S) \times \gamma) \cup \{|\mu_\pm|\}. \]

We are left to prove the reverse inclusion. Suppose $\{u_n\}$ is any sequence in $\X(\gamma)$ with $u_n \to |\mu| \in
\partial \X(\gamma)$. We wish to show that $|\mu| \in \hat\Phi(\partial \C(S) \times \gamma) \cup \{
|\mu_\pm|\}$. By definition of $\X(\gamma)$ there exists $\{(v_n,x_n)\} \subset \C(S) \times \gamma$ with
$\Phi(v_n,x_n) = u_n$ for
all $n$.   There are two cases to consider.\\

\noindent {\bf Case 1.} $\{x_n\} \subset K$, for some compact arc $K \subset \gamma$.\\

After passing to a subsequence if necessary $x_n \to x \in K$.  By Proposition \ref{contfailure}, we can assume that
$v_n$ accumulates on $\partial \C(S)$.  So, after passing to yet a further subsequence if necessary, we can assume that
$v_n \to |\lambda| \in \partial \C(S)$.  Then by continuity of $\hat \Phi$ (Proposition \ref{psiphi2}) we have
\[ |\mu| = \lim_{n \to \infty} \Phi(v_n,x_n) = \hat\Phi(|\lambda|,x) \in \hat\Phi(\partial \C(S) \times \gamma).\]

\noindent {\bf Case 2.} After passing to a subsequence $x_n \to x$, where $x$ is one of the endpoints of $\gamma$ in
$\partial \mathbb H$.\\

Note that $x \in \ATF$ since $p(\gamma)$ is filling. Indeed, $x$ is either the attracting or repelling fixed point of
$\delta$.  Without loss of generality, we assume it is the attracting fixed point.  Now suppose $\gamma_1$ is any
$\pi_1(S)$ translate which nontrivially intersects $\gamma$.  Thus $\{ \delta^n(\gamma_1)\}$ nests down on $x$, and
hence
\[ \bigcap_{n=1}^\infty \overline{\H^+(\delta^n(\gamma_1))} = \bigcap_{n=1}^\infty \delta^n(\overline{\H^+(\gamma_1)})\]
consists of the single point $|\mu_+|$, the attracting fixed point of the pseudo-Anosov $\delta$. After passing to a
further subsequence if necessary, we can assume $x_n \in H^+(\delta^n(\gamma_1))$.  Therefore, $\Phi(v_n,x_n) \in
\H^+(\delta^n(\gamma_1))$, and hence
\[ |\mu| = \lim_{n \to \infty} \Phi(v_n,x_n) = |\mu_+|,\]
completing the proof of Lemma \ref{boundary Xgamma}.
\end{proof}

\begin{proof}[Proof of Proposition \ref{disjointhoods}]  We fix $x,y \in \ATF$ and $\epsilon$ the geodesic between them.  We write $\gamma_x$ and $\gamma_y$ to denote $\pi_1(S)$--translates of $\gamma$ for which $\overline{H^+(\gamma_x)}$ and $\overline{H^+(\gamma_y)}$ define disjoint neighborhoods of $x$ and $y$, respectively.  We must show that $p(\epsilon)$ is simple if and only if $\partial \H^+(\gamma_x) \cap \partial \H^+(\gamma_y) \neq \emptyset$ for all such $\gamma_x$ and $\gamma_y$.

First, suppose $p(\epsilon)$ is simple.  The closure of $p(\epsilon)$ is a lamination $\mathcal L$ \cite{CB}.  Since
$x,y \in \ATF$, $\mathcal L$ must contain some $|\lambda| \in \EL(S)$ as the sublamination obtained by
discarding isolated leaves.  Therefore $\epsilon$ is either a leaf of $p^{-1}(|\lambda|)$ or a diagonal for some complementary
polygon of $p^{-1}(|\lambda|)$.

It follows from Lemma \ref{injective comp reg} that if $x' \in \gamma_x \cap \epsilon$ and $y' \in \gamma_y \cap
\epsilon$, then $\hat\Phi(|\lambda|,x') = \hat \Phi(|\lambda|,y')$.  Appealing to Lemma \ref{boundary Xgamma} we have
\begin{eqnarray*} \emptyset  & \neq &\hat\Phi(\{|\lambda|\} \times \gamma_x) \cap \hat\Phi(\{|\lambda|\}
\times \gamma_y)\\ &  \subset & \partial \X(\gamma_x) \cap \partial \X(\gamma_y)\\
& \subset & \partial \H^+(\gamma_x) \cap \partial \H^+(\gamma_y) \end{eqnarray*} as required. In fact, it is worth
noting that by Lemma \ref{injective comp reg}, $\hat\Phi(\{|\lambda|\} \times \epsilon)$ is a single point which lies
in $\partial \H^+(\gamma_x) \cap \partial \H^+(\gamma_y)$ for all allowed choice of $\gamma_x$ and $\gamma_y$, and is therefore equal to $\overline\Phi_v(x) =
\overline\Phi_v(y)$.

Before we prove the converse, suppose $\gamma_1$ and $\gamma_2$ are two translates of $\gamma$ for which $H^+(\gamma_1)
\subset H^-(\gamma_2)$ and $H^+(\gamma_2) \subset H^-(\gamma_1)$.  Then we have
\[ \partial \H^+(\gamma_1) \subset \partial \H^-(\gamma_2) \quad \mbox{ and } \quad \partial \H^+(\gamma_2) \subset \partial
\H^-(\gamma_1).\] Therefore, by Proposition \ref{sep half infinity}, it follows that
\[\partial \H^+(\gamma_1) \cap \partial \H^+(\gamma_2) = \partial \X(\gamma_1) \cap \partial
\X(\gamma_2).\]

Further suppose that $\gamma_1 \neq \gamma_2$, so that fixed points of $\delta_1$ and $\delta_2$  (elements generating
the stabilizers of $\gamma_1$ and $\gamma_2$, respectively) are disjoint in $\partial \C(S,z)$.  If
\[\partial \H^+(\gamma_1) \cap \partial \H^+(\gamma_2) \neq \emptyset\]
then by Proposition \ref{boundary Xgamma} there exists $x_1 \in \gamma_1$ and $x_2 \in \gamma_2$ and
$|\lambda_1|,|\lambda_2| \in \partial \C(S)$ for which $\hat\Phi(|\lambda_1|,x_1) = \hat\Phi(|\lambda_2|,x_2)$.
According to Lemma \ref{injective comp reg}, we have $|\lambda_1| = |\lambda_2|$, and $x_1$ and $x_2$ lie on the same
leaf, or in the closure of the same complementary region of $|\lambda_1|$.  In particular, there is a biinfinite
geodesic contained in a complementary region or leaf of $p^{-1}(|\lambda_1|)$ which meets both $\gamma_1$ and
$\gamma_2$.

We now proceed to the proof of the converse.  Let $\{ \gamma_{n,x} \}$ and $\{\gamma_{n,y}\}$ be sequences of
$\pi_1(S)$--translates of $\gamma$ which nest down on $x$ and $y$, respectively.  We suppose that
\[ \partial \H^+(\gamma_{n,x}) \cap \partial \H^-(\gamma_{n,y}) \neq \emptyset\]
for all $n \geq 0$, and prove that $p(\epsilon)$ is simple on $S$.  By the discussion in the preceding two paragraphs
there exists a sequence of laminations $\{|\lambda_n|\} \subset \partial \C(S)$ so that $\gamma_{x,n}$ and
$\gamma_{y,n}$ both meet a leaf or complementary polygon of $p^{-1}(|\lambda_n|)$. It follows that there is a sequence
of geodesics $\{ \epsilon_n \}$ in $\mathbb H$ for which $p(\epsilon_n)$ is simple on $S$, and $\epsilon_n \cap
\gamma_{x,n} \neq \emptyset \neq \epsilon_n \cap \gamma_{y,n}$. The limit $\epsilon$ of $\{ \epsilon_n\}$ has endpoints
$x$ and $y$.  Also $p(\epsilon)$ is simple as it is the limit of simple geodesics \cite{CB}.
\end{proof}

The following is now immediate from Proposition \ref{disjointhoods} and its proof.

\begin{corollary} \label{points identified}
Given distinct $x,y \in \ATF$ then $\partial \Phi(x) = \partial \Phi(y)$ if and only if $x$ and $y$ are ideal endpoints
of a leaf (or ideal vertices of a complementary polygon) of $p^{-1}(|\lambda|)$ for some $|\lambda| \in \partial
\C(S)$. \qed
\end{corollary}

%%%%%%%%%%%%%%%%%%%%%%%%%
\subsection{Surjectivity.}
%%%%%%%%%%%%%%%%%%%%%%%%

In this section, we prove that our map $\partial \Phi$ is surjective.

According to Birman--Series \cite{birmanseries}, the union of geodesics
\[ \overline{\bigcup_{v \in \C^0(S)} v} \]
is nowhere dense in $S$.  We fix an $\epsilon > 0$, and assume that our chosen constants $\{\epsilon(v)\}_{v \in
\C^0(S)}$ are sufficiently small so that
\[S - \bigcup_{v \in \C^0(S)} N(v)\]
is $\epsilon$--dense.  It follows that $\epsilon(v) \leq \epsilon$ for all $v \in \C^0(S)$.
\begin{lemma} \label{L:nearly weak convex} Suppose $(v_1,x_1),(v_2,x_2) \in \C^0(S) \times \mathbb H$ with $\Phi(v_i,x_i)
= u_i$ a vertex in $\C(S,z)$ for $i = 1,2$.
Then there is a path
\[ \G = ({\bf v} ,{\bf x}):[a,b] \to \C(S) \times \mathbb H \]
such that $\Phi \circ \G$ is a geodesic from $u_1$ to $u_2$ and ${\bf x}$ connects $x_1$ to
$x_2$ with image contained in the $2\epsilon$--neighborhood of a geodesic in $\mathbb H$.
\end{lemma}
\begin{proof}
For each $i = 1,2$ we can find $x_i'$ in the same component of $S - N^\circ(v_i)$ as $x_i$ within $\epsilon$ of $x_i$
such that $x_1'$ and $x_2'$ are contained in some geodesic $\gamma'$ which projects to a filling closed geodesic in $S$
(the pairs of endpoints of such geodesics is dense in $\partial \mathbb H \times \partial \mathbb H$). Then
$\Phi(v_i,x_i)= \Phi(v_i,x_i')$ for $i = 1,2$.  Moreover, the geodesic from $x_1'$ to $x_2'$ is within $\epsilon$ of
$x_1$ and $x_2$. Suppose we can find $\G' = ({\bf v}',{\bf x}')$ so that $\Phi \circ \G'$ is a geodesic from $u_1$ to
$u_2$ and ${\bf x}'$ connects $x_1'$ to $x_2'$ with image contained in the $\epsilon$--neighborhood of a geodesic
containing $x_1'$ and $x_2'$.  Then we can take $\G = ({\bf v},{\bf x})$ to be such that ${\bf v} = {\bf v}'$ and ${\bf
x}$ first runs from $x_1$ to $x_1'$, then traverses ${\bf x}'$, and finally runs from $x_2'$ to $x_2$ (all
appropriately reparameterized). This will then provide the desired path proving the lemma.

To construct $\G'$, we suppose for the moment that $\{\epsilon(v)\}_{v \in \C^0(S)}$ have been chosen so that any arc
of $\gamma' \cap p^{-1}(N(v))$ is essential.  With this assumption, Proposition \ref{weakconvex} applied to $\gamma'$
implies that $\X(\gamma')$ is weakly convex.  Now connect $u_1$ and $u_2$ by a geodesic edge path within $\X(\gamma')$
with vertex set $\{u_1=w_1,w_2,w_3,...,w_k = u_2\}$.

Let $v_i = \Pi(w_i)$.  We observe that for every $i = 1,...,k$,
\[ \Phi^{-1}(w_i) \cap (\C(S) \times \gamma') = \{v_i \} \times \alpha_i\]
where $\alpha_i$ is an arc of $\gamma' \cap (\mathbb H - p^{-1}(N(v_i)))$ and is in particular
connected.  It follows from the construction of $\Phi$ that the edges $[w_i,w_{i+1}]$, for $i = 1,...,k-1$ are
images of paths in $\C(S) \times \gamma'$ which we denote $a_i = (b_i,c_i)$.  Explicitly, if $v_i = v_{i+1}$, then
$b_i$ is constant and equal to $v_i=v_{i+1}$, and $c_i$ traverses an arc of $\gamma' \cap p^{-1}(N(v_i))$.  If
$v_i \neq v_{i+1}$, then $b_i$ traverses the edge $[v_i,v_{i+1}]$ and $c_i$ is constant.

We can now define $\G' = ({\bf v}',{\bf x}')$ as follows.
\begin{enumerate}
\item Begin by holding ${\bf v}'$ constant equal to $u_1 = w_1$ and let ${\bf x}'$ traverse from $x_1'$ to the initial point of $c_1$ inside $\alpha_1 \subset \gamma'$.
\item Next, traverse $a_1$.
\item After that, hold ${\bf v}'$ constant again and let ${\bf x}'$ traverse from the terminal point of $c_1$ to the initial point of $c_2$ inside $\alpha_2 \subset \gamma'$.
\item We can continue in this way, for $i = 2,...,k-2$ traversing $a_i$, then holding ${\bf v}'$ constant and letting ${\bf x}'$ go from the terminal point of $c_i$ to the initial point of $c_{i+1}$ inside $\alpha_{i+1} \subset \gamma'$.
\item We complete the path by traversing $a_{k-1}$, then holding ${\bf v}'$ constant and letting ${\bf x}'$ traverse the path from the terminal point of $c_{k-1}$ to
$x_2'$ inside $\alpha_k \subset \gamma'$.
\end{enumerate}
By construction, the projection of this path $\Phi \circ \G'$ onto the first coordinate is the geodesic
from $u_1$ to $u_2$ that we started with (although it stops and is constant at each of the vertices for some interval
in the domain of the parametrization). Moreover, ${\bf x}'$ is contained in $\gamma'$ and connects $x_1'$ to
$x_2'$, so therefore stays within a distance zero of the geodesic from $x_1'$ to $x_2'$, as required.

The proof so far was carried out under the assumption that for every $v \in \C^0(S)$, every arc of $\gamma' \cap
N(v)$ enters and exits the component of $N(v)$ which it meets in different boundary
components.  If this is not true, then first shrink all $\epsilon(v)$ to numbers $\epsilon'(v) < \epsilon(v)$ so that
it is true, construct the path as above, and call it $\G''=({\bf v}'',{\bf x}'')$.  Note
that the numbers $\{\epsilon'(v)\}_{v \in \C^0(S)}$ determine a new map $\Phi':\C(S) \times \mathbb H \to \C(S,z)$, and
$\Phi' \circ \G''$ is a geodesic.  With respect to the original map $\Phi$, $\widetilde \nu''$ is almost
good enough for our purposes. The only problem is that $\Phi \circ \G''$ may now no longer be a geodesic:
If there is some interval in the domain in which ${\bf v}''$ is constant equal to $v$ and ${\bf x}''$
enters and exits a component $p^{-1}(N(v))$ from the same side, then $\Phi \circ \G''$ will
divert from being a geodesic by running (less than half way) into an edge of $\Pi^{-1}(v)$ and running back out.  We
modify $\G''$ to the desired path $\G'$, by pushing ${\bf x}''$ outside of
$p^{-1}(N(v))$ whenever this happens, thus changing it by at most $\epsilon(v) \leq \epsilon$.  The
resulting path $\G'$ has ${\bf v}' = {\bf v}''$ and ${\bf x}'$ still connects
$x_1'$ to $x_2'$ and stays within $\epsilon$ of $\gamma'$, as required.
\end{proof}

Surjectivity of $\partial \Phi$ requires that every point of $\partial \C(S,z)$ is the limit of $\Phi_v(r)$ for some $v \in \C^0(S)$ and some ray
$r \subset \mathbb H$ ending at a point of $\ATF$.  The following much weaker conclusion is easier to arrive at, and
will be used in the proof of surjectivity.

\begin{lemma} \label{limiteverywhere}
For any $v \in \C^0(S)$, $\partial \C(S,z) \subset \overline{\Phi_v(\mathbb H)}$.
\end{lemma}
\begin{proof}
First, note that since $\pi_1(S) < \Mod(S,z)$ is a normal, infinite subgroup the limit set in $\PML(S,z)$ (in the sense
of \cite{mcpapa}) is all of $\PML(S,z)$.  In particular, the closure of any $\pi_1(S)$--equivariant embedding $\mathbb
H \subset \T(S,z)$ in the Thurston compactification of Teichm\"uller space meets the boundary $\PML(S,z)$ in all of $\PML(S,z)$.  In
particular, for any $\mu \in \PFL$, there is a sequence of points $x_n \in \mathbb H$ limiting to $\mu$.

The systol map $\sys:\T(S,z) \to \C(S,z)$ restricts to a $\pi_1(S)$--equivariant map from $\mathbb H$ to $\C(S,z)$,
which is therefore a bounded distance from $\Phi_v$.   Again appealing to Klarreich's work \cite{klarreich-el}, it
follows that $\sys$ extends continuously to $\PFL(S,z)$, and hence $\sys(x_n) \to |\mu| \in \EL(S,z) \cong
\partial \C(S,z)$. Therefore $\Phi_v(x_n) \to |\mu|$.  Since $\mu$ was arbitrary, every point of
$\partial \C(S,z)$ is a limit of a sequence in $\Phi_v(\mathbb H)$, and we are done.
\end{proof}

Given an arbitrary sequence $\{x_n\}$ in $\mathbb H$, we need to prove the following.
\begin{proposition} \label{notangents}
If $\displaystyle{\lim_{n \to \infty} x_n = x \in \partial \mathbb H - \ATF}$, then $\Phi_v(x_n)$ does not converge to
a point of $\partial \C(S,z)$.
\end{proposition}

One case of this proposition requires a different proof, and we deal with this now.

\begin{lemma} \label{notangentsbaby}
If $\{x_n\}$ and $x$ are as in Proposition 3.12 and $x$ is the endpoint of a lift of a closed geodesic on $S$, then
$\Phi_v(x_n)$ does not converge to a point of $\partial \C(S,z)$
\end{lemma}
\begin{proof}  Under the hypothesis of the lemma, there is an element $\eta \in \pi_1(S)$ with $x$ as the attracting fixed
point.  Moreover, because $x \not \in \ATF$, the geodesic representative of this element of $\pi_1(S)$ does not fill
$S$.  Therefore, the associated mapping class is reducible (see \cite{kra}).

Let $\gamma_0$ be a $\pi_1(S)$--translate of $\gamma$ such that $\gamma_0$ separates $x$ from the repelling fixed point of $\eta$.
Then $\{\eta^n(\gamma_0)\}$ nest down on $x$.  It follows that after passing to a subsequence
(which we continue to denote $\{x_n\}$) we have
\[ x_n \in H^+(\eta^n(\gamma_0)) = \eta^n(H^+(\gamma_0)).\]
Appealing to the $\pi_1(S)$--equivariance of $\Phi$ we have
\[\Phi_v(x_n) = \Phi(v,x_n) \in \H^+(\eta^n(\gamma_0)) = \eta^n(\H^+(\gamma_0)).\]

Suppose now that $\Phi_v(x_n)$ converges to some element $|\mu| \in \partial \C(S,z)$.  It follows that
\[ |\mu| \in \bigcap_{n=1}^\infty \eta^n (\overline{\H^+(\gamma_0)}).\]
However, any such $|\mu|$ is invariant under $\eta$ and since $\eta$ is a reducible mapping class it fixes no point
of $\partial \C(S,z)$.  This contradiction implies $\Phi_v(x_n)$ does not converge to any $|\mu| \in
\partial \C(S,z)$, as required.
\end{proof}

\begin{proof}[Proof of Proposition \ref{notangents}]
Recall that $\Phi_v(x) = \Phi(v,x)$.
Suppose, contrary to the conclusion of the proposition, that
\[
\lim_{n \to \infty} \Phi(v,x_n) = |\mu| \in \EL(S,z) \cong \partial \C(S,z).
\]
We begin by finding another sequence which also converges to $|\mu|$ to which we can apply the techniques developed so
far. Since $x \not \in \ATF$ the surface $Y$ filled by $x$ is strictly contained in $S$.  By Lemma \ref{notangentsbaby}
we may assume that $Y$ is not an annulus.  Let $r \subset \mathbb H$ be a ray ending at $x$ so that $r$ is contained in
a component $\widetilde Y$ of $p^{-1}(Y)$ and so that $p(r)$ fills $Y$.

We pass to a subsequence (which we continue to denote $\{x_n\}$) with the property that for every $k > 0$, the
geodesic segment $\beta_k$ connecting $x_{2k}$ to $x_{2k+1}$ passes within some fixed distance, say distance $1$, of $r$
and so that furthermore
\[
\beta_k \cap \widetilde Y \neq \emptyset.
\]

Now fix any $k > 0$ and let $\Phi \circ \G^k:[a_k,b_k] \to \C(S,z)$ be a geodesic from $\Phi(v,x_{2k})$ to
$\Phi(v,x_{2k+1})$ where $\G^k = ({\bf v}^k,{\bf x}^k)$ is given by Lemma \ref{L:nearly weak convex}.  The path ${\bf
x}^k$ connects $x_{2k}$ to $x_{2k+1}$ and has image within $2 \epsilon$ of a geodesic in $\mathbb H$ which must also
pass within a uniformly bounded distance of $r$ (in fact, it passes within a distance $1+2\epsilon$).

Choosing a sufficiently thin subsequence $\{x_n\}$, we may assume that $\beta_k$ spends a very long time in $\widetilde
Y$.  Doing this ensures that the image of ${\bf x}^k$ nontrivially intersects $\widetilde Y$. Let $t_k \in [a_k,b_k]$
be any time where ${\bf x}^k$ meets $\widetilde Y$.  Then set
\[y_k = {\bf x}^k(t_k) \in \widetilde Y \quad \mbox{and} \quad v_k = {\bf v}^k(t_k) \in \C^0(S)\]
(recall that ${\bf x}^k$ is constant when ${\bf v}^k$ is not, so we can assume that $t_k$ is chosen so that $v_k$ is
indeed a vertex).

Now observe that since $\Phi \circ \G^k([a_k,b_k])$ is a geodesic from $\Phi(v,x_{2k})$ to $\Phi(v,x_{2k+1})$, the sequence
$\{\Phi \circ \G^k(t_k)\} = \{\Phi(v_k,y_k)\}$ also converges to $|\mu|$.  Let us write $u_k = \Phi(v_k,y_k)$.

Next, for each $k > 0$ let $f_k \in \Diff_0(S)$ be such that $\widetilde \ev(f_k) = y_k \in \widetilde Y$.  Since
$\widetilde Y$ is a single component of $p^{-1}(Y)$, we may assume that any two $f_j$ and $f_k$ differ by an isotopy
fixing the complement of the interior of $Y$.  That is, there is a path $f_t \in \Diff_0(S)$ for $t \in [1,\infty)$ such that $y_k = \widetilde\ev(f_k)$ for all positive integers $k$, and so that
\[f_1|_{S - Y} = f_t|_{S-Y}\]
for all $t \in [1,\infty)$.

Let $X = f_1^{-1}(Y)$ and consider the punctured surfaces
\[ Y^\circ = Y - \{f_1(z)\} \quad \mbox{ and } \quad X^\circ = X - \{z\} = f_1^{-1}(Y^\circ).\]
We will be interested in the set of subsurface projections
\[ \{\pi_{X^\circ}(u_k)\} \subset \C'(X^\circ) \]
where $\C'(X^\circ)$ is the arc complex of $X^\circ$; see \cite{masur-minsky2}.  We consider the incomplete metric on
$X^\circ$ for which $f_1: X^\circ \to Y^\circ$ is an isometry where $Y^\circ$ is given the induced path metric inside
of $S$.

\begin{claim} The length of some arc of $\pi_{X^\circ}(u_k)$ tends to infinity.
\end{claim}

Here, length means infimum of lengths over the isotopy class of an arc. The claim implies that there are infinitely
many arcs in the set $\{\pi_{X^\circ}(u_k)\}$ which is impossible if $u_k \to |\mu|$. Thus, to complete the proof of
the proposition, it suffices to prove the claim.

\medskip

\noindent {\it Proof of Claim:} So, to prove that the length of some arc tends to infinity, first suppose that
$\{\pi_Y(v_k)\}$ contains an infinite set.  Then there are arcs $\alpha_k \subset \pi_Y(v_k)$ with $\ell_Y(\alpha_k)
\to \infty$.  Now $f_k^{-1}(\alpha_k)$ is an arc of $\pi_{X^\circ}(u_k)$ and $\ell_{X^\circ}(f_k^{-1}(\alpha_k)) =
\ell_{Y^\circ}(f_1 f_k^{-1}(\alpha_k))$. However, $f_1 f_k^{-1}$ is the identity outside the interior of $Y$, in
particular it is the identity on the boundary of $Y$ and isotopic (forgetting $z$) to the identity in $Y$.  So, we have
\[ \ell_{Y^\circ}(f_1 f_k^{-1}(\alpha_k)) \geq \ell_Y(\alpha_k) \to \infty\]
and hence there is an arc of $\pi_{X^\circ}(u_k)$ with length tending to infinity as required.

We may now suppose that there are only finitely many arcs in the set $\{\pi_Y(v_k)\}$.  By passing to a further
subsequence if necessary, we may assume that $\pi_Y(v_k)$ is constant and equal to a union of finitely many arcs in
$Y$. We fix attention on one arc, call it $\alpha$.  Again, we see that $f_k^{-1}(\alpha)$ is an arc of
$\pi_{X^\circ}(u_k)$ and $\ell_{X^\circ}(f_k^{-1}(\alpha)) = \ell_{Y^\circ}(f_1 f_k^{-1}(\alpha))$ with $f_1 f_t^{-1}$
equal to the identity outside the interior of $Y$ for all $t$.

Writing $h_t = f_t f_1^{-1}$, we are required to prove that $\ell_{Y^\circ}(h_k^{-1}(\alpha))$ tends to infinity as $k
\to \infty$.  Observe that $h_1$ is the identity on $S$ and $h_t$ is the identity outside the interior of $Y$ for all
$t \in [1,\infty)$.  We can lift $h_t$ to $\widetilde h_t$ so that $\widetilde h_1$ is the identity in $\mathbb H$. It
follows from the definition of $\widetilde{\ev}$ that $\widetilde h_k(\widetilde{\ev}(f_1)) = y_k$. Thus, $\widetilde
h_t$ is essentially pushing the point $y = \widetilde{\ev}(f_1) \in \widetilde Y$ along the ray $r$ (at least,
$\widetilde h_k(y) = y_k$ comes back to within a uniformly bounded distance to $r$ for every positive integer $k$,
though it is not hard to see that we can choose $f_t$ so that $\widetilde h_t$ always stays a bounded distance from
$r$).

Now $h_t^{-1}(\alpha)$ can be described as applying the isotopy $h_t$ \textit{backward} to $\alpha$.  Therefore, if we
let $\widetilde \alpha^k$ be the last arc of $p^{-1}(\alpha)$ intersected by the path $\widetilde h_t(y)$ for $t \in
[1,k]$, then we can drag $\widetilde \alpha^k$ backward using the isotopy $\widetilde h_t$ as $t$ runs from $k$ back to
$1$, and the result $\widetilde h_k^{-1}(\widetilde \alpha^k)$ projects down by $p$ to $h_k^{-1}(\alpha)$; see Figure
\ref{nolimitfigure}. Moreover, observe that $\ell_{Y^\circ}(h_k^{-1}(\alpha))$ is at least the sum of the distances
from $y$ to the two boundary components of $\widetilde Y$ containing the end points of $\widetilde \alpha^k$.

\begin{figure}[htb]
\begin{center}
\ \psfig{file=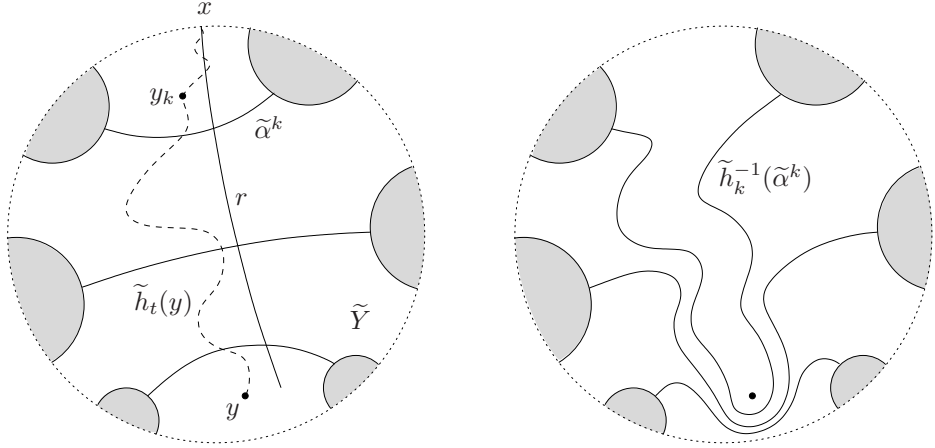,height=2.2truein} \caption{On the left: $r$ inside $\widetilde Y$ (the complement of the
shaded region), the path $\widetilde h_t(y)$ as it goes through $y_k= \widetilde h_k(y)$ and the arc $\widetilde
\alpha^k$. On the right: dragging $\widetilde \alpha^k$ back by $\widetilde h_k^{-1}$.} \label{nolimitfigure}
\end{center}
 \setlength{\unitlength}{1in}
  \begin{picture}(0,0)(0,0)
    \put(1.8,1.55){$\widetilde Y$}
    \put(1.3,2.55){$\widetilde \alpha^k$}
    \put(1.15,1.1){$y$}
    \put(.75,2.75){$y_k$}
    \put(.67,1.65){$\widetilde h_t(y)$}
    \put(3.73,2.3){$\widetilde h_k^{-1}(\widetilde \alpha^k)$}
    \put(1,3.2){$x$}
    \put(1.2,2.2){$r$}
  \end{picture}
\end{figure}

Finally, since $r$ fills $Y$, the distance from $y$ to the boundary components of $\widetilde Y$ containing the
endpoints of $\widetilde \alpha^k$ must be tending to infinity as $k \to \infty$ (otherwise, we would find that $r$ is
asymptotic to one of the boundary components of $\widetilde Y$ which contradicts all tails filling in $Y$). This
implies $\ell_{Y^\circ}(h_k^{-1}(\alpha))$ tends to infinity as $k \to \infty$.  This proves the claim, and so
completes the proof of the proposition.
\end{proof}

We can now prove one of the main technical pieces of Theorem \ref{cannonthurston}.

\pagebreak

\begin{theorem} \label{ctsurjective}
The map
\[ \partial \Phi:\ATF \to \partial \C(S,z)\]
is surjective.
\end{theorem}
\begin{proof} Let $|\mu| \in \partial \C(S,z)$ be an arbitrary point.
According to Lemma \ref{limiteverywhere} there exists a sequence $\{ x_n \} \subset \mathbb H$ with
$$\lim_{n \to \infty} \Phi_v(x_n) = |\mu|.$$
By passing to a subsequence, we may assume that $\{x_n\}$ converges to a point $x \in \partial \mathbb H$. It follows
from Proposition \ref{notangents} that $x \in \ATF$.  Then, by Theorem \ref{cannonthurstonhalf}
$$|\mu| = \lim_{n \to \infty} \Phi_v(x_n) = \overline \Phi_v(x) = \partial \Phi(x).$$
Since $|\mu| \in \partial \C(S,z)$ was an arbitrary point, it follows that $\partial \Phi(\ATF) = \partial \C(S,z)$, and $\partial \Phi$ is surjective.
\end{proof}

%%%%%%%%%%%%%%%%%%%%%%%%%%%%%%%%%%
\subsection{Neighborhood bases.}
%%%%%%%%%%%%%%%%%%%%%%%%%%%%%%%%%%

In this section we find neighborhood bases for points of $\partial \C(S,z)$.

\begin{remark}
For the purposes of the present work, it is more convenient to use the more flexible definition of neighborhood basis.
This is a collection of sets $\{V_j(x)\}_j$ associated to each point $x$ in the space with the property that a subset
$U$ is open if and only if for every $x \in U$, $V_j(x) \subset U$ for some $j$.  Equivalently, the interior of each
$V_j(x)$ is required to contain $x$, and any open set containing $x$ should contain some $V_j(x)$.
\end{remark}

To construct the neighborhood bases, we must distinguish between two types of points of $\ATF$.  We say a point $x \in
\ATF$ is \textit{simple} if there exists a ray $r$ in $\mathbb H$ ending at $x$ for which $p(r)$ is simple.  Otherwise
$x$ is not simple. Equivalently, a point $x \in \ATF$ is simple if and only if there is a lamination $|\lambda| \in
\EL(S)$ such that $x$ is the ideal endpoint of a leaf (or ideal vertex of a complementary polygon) of
$p^{-1}(|\lambda|)$.

\begin{lemma} \label{easyquotient topology}
If $x \in \ATF$ is not simple and $\{ \gamma_n\}$ are $\pi_1(S)$--translates of $\gamma$ which nest down to $x$, then
$\{ \partial \H^+(\gamma_n) \}$ is a neighborhood basis for $\partial \Phi(x)$.
\end{lemma}
\begin{proof}
Since the distance of $\H^+(\gamma_n)$ to any fixed basepoint in $\C(S,z)$ tends to infinity as $n \to \infty$, it
follows that the visual diameter of $\overline{\H^+(\gamma_n)}$ measured from any base point also tends to zero.  Thus,
for any neighborhood $U$ of $\partial \Phi(x)$ in $\partial \C(S,z)$, there exists $N > 0$ so that for all $n \geq N$,
$\partial \H^+(\gamma_n) \subset U$.

We must prove $\partial \Phi(x)$ is in $\INT(\partial \H^+(\gamma_n))$ for all $n$. We already know that
\[ \bigcap \partial \H^+(\gamma_n) = \{ \partial \Phi(x)\} \]
and in particular, $\partial \Phi(x) \in \partial \H^+(\gamma_n)$ for all $n$.
Therefore, it suffices to prove that for any $n$, there exists $m > n$ so that
\[ \partial \H^+(\gamma_m) \subset \INT(\partial \H^+(\gamma_n)).\]

It follows from Proposition \ref{sep half infinity} and the fact that $\partial \H^-(\gamma_n)$ is a closed subset of
$\partial \C(S,z)$ that
\[ \partial \H^+(\gamma_n) - \partial
\X(\gamma_n)  = \partial \C(S,z) - \partial \H^-(\gamma_n) \subset \INT(\partial \H^+(\gamma_n)).\] For any $m > n$, we
also know
\[ \partial \H^+(\gamma_m) \subset \partial \H^+(\gamma_n).\]
Thus, if we can find $m > n$ so that
\[ \partial \X(\gamma_n) \cap \partial \X(\gamma_m) = \emptyset,\]
then appealing to Proposition \ref{sep half infinity} again, it will follow that
\[ \partial \H^+(\gamma_m ) \subset \partial \H^+(\gamma_n) - \partial \X(\gamma_n) \subset \INT(\partial \H^+(\gamma_n)),\]
as required.

If for all $m > n$ we have $\partial \X(\gamma_m) \cap \partial \X(\gamma_n) \neq \emptyset$, then a similar proof to
that given for Proposition \ref{disjointhoods} shows that $x$ is a simple point which is a contradiction.
\end{proof}

The above lemma gives a neighborhood basis for $\partial \Phi(x)$ when $x \in \ATF$ not a simple point.  The next lemma
describes a neighborhood basis $\partial \Phi(x)$, where $x$ is a simple point.

Suppose $x_1,x_2$ are endpoints of a nonboundary leaf of $p^{-1}(|\lambda|)$ or $x_1,...,x_k$ are points of a
complementary polygon of some $p^{-1}(|\lambda|)$ for some $|\lambda| \in \EL(S)$.   We treat both cases simultaneously
referring to these points as $x_1,...,x_k$.  As already noted in the proof of Theorem \ref{cannonthurston}, $\partial
\Phi(x_1) = ... = \partial \Phi(x_k)$, and the $\partial \Phi$--image of any simple point has this form.

\begin{lemma} \label{quotient topology}
If $x_1,...,x_k$ are as above, and $\{\gamma_{1,n}\}$,...,$\{\gamma_{k,n}\}$ are sequences of $\pi_1(S)$--translates of
$\gamma$ with $\{ \gamma_{j,n} \}$ nesting down to $x_j$ for each $j = 1,...,k$, then
\[ \left\{ \partial \H^+(\gamma_{1,n}) \cup \cdots \cup \partial \H^+(\gamma_{k,n}) \right\}_{n=1}^\infty\]
is a neighborhood basis for $\partial \Phi(x_1) = ... = \partial \Phi(x_k)$.
\end{lemma}
\begin{proof}
Let $|\mu| = \partial \Phi(x_1) = ... = \partial \Phi(x_k)$. As in the proof of the previous lemma, the sets in the
proposed neighborhood basis have visual diameter tending to zero as $n \to \infty$.

Since $|\mu| \in \partial \H^+(\gamma_{j,n})$ for all $j = 1,...,k$, we clearly have
\[ |\mu| \in \partial \H^+(\gamma_{1,n}) \cup \cdots \cup \partial \H^+(\gamma_{k,n}). \]
Thus, we are required to show that $|\mu|$ is an interior point of this set.

This is equivalent to saying that for any sequence $\{|\mu_m|\} \subset \partial \C(S,z)$ converging to $|\mu|$,
and every positive integer $n$, there exists $M > 0$ so that for all $m \geq M$,
\begin{equation} \label{almost there}
|\mu_m| \in \partial \H^+(\gamma_{1,n}) \cup \cdots \cup \partial \H^+(\gamma_{k,n}).\end{equation}

So, let $\{|\mu_m|\} \subset \partial \C(S,z)$ be a sequence converging to $|\mu|$ and $n$ a positive integer.
Choose any sequence $\{y_m\} \subset \ATF$ so that $\partial \Phi(y_m) = |\mu_m|$ (such a sequence exists by
surjectivity of $\partial \Phi$). We wish to show that any accumulation point of $\{y_m\}$ is one of the points
$x_1,...,x_k$. For then, we can find an $M>0$ so that for all $m \geq M$
\[y_m \in \overline{H^+(\gamma_{1,n})} \cup ... \cup \overline{H^+(\gamma_{k,n})}\]
and hence \eqref{almost there} holds.

To this end, we pass to a subsequence so that $y_m \to x \in \partial \mathbb H$. Choosing sequences converging to
$y_m$ for all $m$ and applying a diagonal argument, we see that there is a sequence $\{q_m\} \subset \mathbb H$ with
$\displaystyle{\lim_{m \to \infty} q_m = x}$ and $\displaystyle{\lim_{m \to \infty} \Phi_v(q_m)} = |\mu|$. From
Proposition \ref{notangents} we deduce that $x \in \ATF$.

Now, if $x \in \{x_1,...,x_k\}$ then we are done.  Suppose not.  Then the geodesic $\epsilon_j$ from $x$ to $x_j$ has
$p(\epsilon_j)$ non-simple for all $j$.  Proposition \ref{disjointhoods} guarantees $\pi_1(S)$--translates
$\gamma_x,\gamma_{1,n},...,\gamma_{k,n}$ of $\gamma$ defining neighborhoods
$\overline{H^+(\gamma_x)},\overline{H^+(\gamma_{1,n})},...,\overline{H^+(\gamma_{k,n})}$ of $x,x_1,...,x_k$,
respectively for which
\[ \partial \H^+(\gamma_x) \cap \partial \H^+(\gamma_{j,n}) = \emptyset \]
for all $j = 1,...,k$.  Since $\partial \Phi$ is continuous, we have
\[ \lim_{m \to \infty} \partial \Phi(y_m) = \lim_{m \to \infty} |\mu_m| = |\mu| = \partial \Phi(x) \in \partial \H^+(\gamma_x).\]
This is impossible since $|\mu| \in \partial \H^+(\gamma_{j,n})$ for all $j = 1,...,k$.  Therefore, $x = x_j$ for some
$j$, and the proof is complete.
\end{proof}

We are now ready to prove
\smallskip

\noindent {\bf Theorem} \ref{cannonthurston} (Universal Cannon--Thurston map). \textit{For any $v \in \C^0(S)$, the map
$\Phi_v: \mathbb H \to \C(S,z)$ has a unique continuous $\pi_1(S)$--equivariant extension
$$\overline\Phi_v: \mathbb H \cup \ATF \to \overline \C(S,z).$$
The map $\partial \Phi = \overline \Phi_v|_{\ATF}$ does not depend on $v$ and is a quotient map onto $\partial
\C(S,z)$.  Given distinct points $x,y \in \ATF$, $\partial \Phi(x) = \partial \Phi(y)$ if and only if $x$ and $y$ are
ideal endpoints of a leaf (or ideal vertices of a complementary polygon) of the lift of an ending lamination on $S$.}
\smallskip

\begin{proof}
By Theorem \ref{cannonthurstonhalf}, Corollary \ref{points identified} and Theorem \ref{ctsurjective} all that remains
is to prove that $\partial \Phi$ is a quotient map.  To see this, we need only show that $E \subset
\partial \C(S,z)$ is closed if and only if $F = \partial \Phi^{-1}(E)$ is closed.  Since $\partial \Phi$
is continuous, it follows that if $E$ is closed, then $F$ is closed.

Now, suppose that $F$ is closed.  To show that $E$ is closed, we let $|\mu_n| \to |\mu|$ with $\{|\mu_n|\} \subset E$
and we must check that $|\mu| \in E$. By Lemmas \ref{easyquotient topology} and \ref{quotient topology}, after passing
to a subsequence if necessary, there is a sequence $\{\gamma_n\}$ nesting down on some point $x \in \partial
\Phi^{-1}(|\mu|)$ with $|\mu_n| \in
\partial \H^+(\gamma_n)$.  Let $x_n \in \partial \Phi^{-1}(|\mu_n|) \subset F$ be such that $x_n \in \partial H^+(\gamma_n)$. It
follows that $x_n \to x$, so since $F$ is closed, $x \in F$. Therefore, $\partial \Phi(x) = |\mu| \in E$, as required.
Thus, $E$ is closed, and $\partial \Phi$ is a quotient map.
\end{proof}

%%%%%%%%%%%%%%%%%%%%%%%%%%%%%%%%%%%%%%%%%%%%%%%%%%%%%%%%%%%%%%%%%%%%
\subsection{$\Mod(S,z)$--equivariance.} \label{S:mod equivariance}
%%%%%%%%%%%%%%%%%%%%%%%%%%%%%%%%%%%%%%%%%%%%%%%%%%%%%%%%%%%%%%%%%%%

We now prove
\smallskip

\noindent \textbf{Theorem \ref{full equivariance}.}
\textit{The quotient map
\[ \partial \Phi: \ATF \to \partial \C(S,z) \]
constructed in Theorem \ref{cannonthurston} is equivariant with respect to the action of $\Mod(S,z)$.}

\smallskip
\begin{proof}  It suffices to prove
\[\partial \Phi(\phi(x)) = \phi(\partial \Phi(x)).\]
for every $\phi \in \Mod(S,z)$ and a dense set of points $x \in \ATF$.

Let $\gamma' \subset \mathbb H$ be a geodesic for which $p(\gamma')$ is a filling closed geodesic in $S$ and let
$\delta' \in \pi_1(S)$ be the generator of the infinite cyclic stabilizer of $\gamma'$.  Let $x \in \ATF$ denote the
attracting fixed point of $\delta'$.  As previously discussed, according to Kra \cite{kra}, $\delta'$ represents a
pseudo-Anosov mapping class in $\Mod(S,z)$, and the $\pi_1(S)$--equivariance of $\partial \Phi$ implies $\partial
\Phi(x)$ is the attracting fixed point for $\delta'$ in $\partial \C(S,z)$.

Now, given any $\phi \in \Mod(S,z)$, note that $\phi(x)$ is the attracting fixed point of $\phi \circ \delta' \circ
\phi^{-1}$ in $\ATF$, and $\phi(\partial \Phi(x))$ is the attracting fixed point for $\phi \circ \delta' \circ
\phi^{-1}$ in $\partial \C(S,z)$.  Appealing to the $\pi_1(S)$--equivariance again, we see that $\partial \Phi$ must
take $\phi(x)$ to $\phi(\partial \Phi(x))$.  That is
\[ \partial \Phi(\phi(x)) = \phi(\partial \Phi(x)). \]
Since the set of endpoints of such geodesics is dense in $\ATF$, this completes the proof.
\end{proof}

%%%%%%%%%%%%%%%%%%%%%%%%%%%%%%%%%%%%%%%%%%%%%%%%%
%\subsection{Universality of $\overline \Phi_v$}
%%%%%%%%%%%%%%%%%%%%%%%%%%%%%%%%%%%%%%%%%%%%%%%%%

%Proposition \ref{disjointhoods} together with Theorem \ref{cannonthurston} provides a description of $\partial
%\C(S,z)$ as a quotient space of $\ATF$.

%\begin{theorem} The map
%\[ \partial \Phi_v: \ATF \to \partial \C(S,z) \]
%is a quotient map.  Moreover, two points $x,y \in \ATF$ have $\partial \Phi_v(x) = \partial \Phi_v(y)$ if and only if
%there exists an ending lamination $|\lambda| \in \EL(S)$ so that $x$ and $y$ are on the endpoint of a leaf of
%$p^{-1}(|\lambda|)$ or ideal vertices of a complimentary region
%\end{theorem}

%%%%%%%%%%%%%%%%%%%%%%%%%%%%%%%%%
\section{Local path-connectivity.} \label{S:local pc}
%%%%%%%%%%%%%%%%%%%%%%%%%%%%%%%%%

The following, together with Lemma \ref{quotient topology} will prove Theorem \ref{localpathconn}.

\begin{lemma} \label{halfspaceconn}
$\partial \H^+(\gamma)$ is path-connected.
\end{lemma}
\begin{proof}
Fix any $|\lambda| \in \EL(S)$.  According to Proposition \ref{psiphi2}, $\hat \Phi$ is continuous, so we have a
path-connected subset
\[ \hat \Phi(\{ |\lambda|\} \times H^+(\gamma)) \subset \partial \H^+(\gamma). \]

Now let $|\mu| \in \partial \H^+(\gamma)$  be any point.  We will construct a path in $\partial \H^+(\gamma)$
connecting a point of $\hat \Phi(\{ |\lambda|\} \times H^+(\gamma))$ to $|\mu|$.  This will suffice to prove the lemma.

According to Theorem \ref{cannonthurston} there exists $x \in \ATF$ so that $\partial \Phi(x) = |\mu|$.  Let $r:[0,1)
\to H^+(\gamma)$ be a ray with
\[\lim_{t \to 1} r(t) = x. \]

Let $\{\gamma_n\}$ be a sequence of $\pi_1(S)$--translates of $\gamma$ which nest down on $x$.  We assume, as we may,
that $\gamma_1 = \gamma$.  Therefore, there is a sequence $t_1 < t_2 <...$ with $\lim_{n \to \infty}t_n = 1$ and
\[r([t_n,1)) \subset H^+(\gamma_n) \]
and hence again by Proposition \ref{psiphi2}
\[\hat \Phi(\{|\lambda|\} \times r([t_n,1)) \subset \hat \Phi(\{|\lambda|\} \times H^+(\gamma_n)) \subset \partial \H^+(\gamma_n). \]

Recall that, by definition, $\partial \Phi(x)$ is the unique point of intersection
\[\bigcap_{n=1}^\infty \partial \H^+(\gamma_n),\]
and hence
\[\lim_{t \to 1} \hat \Phi(|\lambda|,r(t)) = |\mu|. \]

Therefore, we can extend $R_{|\lambda|}(t) = \hat \Phi(|\lambda|,r(t))$ to a continuous map
\[R_{|\lambda|}:[0,1] \to \partial \H^+(\gamma)\]
with $R_{|\lambda|}(0) \in \hat \Phi(\{ |\lambda|\} \times H^+(\gamma))$ and $R_{|\lambda|}(1) = |\mu|$. This is the
required path completing the proof.
\end{proof}

We now prove
\smallskip

\noindent \textbf{Theorem \ref{localpathconn}.} \textit{The Gromov boundary $\partial \C(S,z)$ is path-connected and
locally path-connected.}

\smallskip

\begin{proof} From Lemma \ref{halfspaceconn}, we see that every set of the form $\partial \H^+(\gamma_0)$ is
path-connected for any $\pi_1(S)$--translate $\gamma_0$ of $\gamma$.  According to Lemmas \ref{easyquotient topology}
and \ref{quotient topology} there is a bases for the topology consisting of these sets (and finite unions of these sets
which all share a point), this proves local path-connectivity. Path-connectivity follows from Lemma \ref{halfspaceconn}
and Proposition \ref{sep half infinity}.
\end{proof}

\bibliography{uctmaps}
\bibliographystyle{alpha}

\bigskip

\noindent Department of Mathematics, University of Illinois, Urbana--Champaign, IL 61801, USA \newline \noindent
\texttt{clein@math.uiuc.edu}

\bigskip

\noindent School of Mathematical Sciences, RKM Vivekananda University, Belur Math, WB 711202, India \newline \noindent
\texttt{mahan@rkmvu.ac.in}

\bigskip

\noindent Department of Mathematics, University of Warwick, Coventry CV4 7AL UK
\newline \noindent \texttt{s.schleimer@warwick.ac.uk }

\end{document}